\documentclass[aop, preprint]{imsart}
\usepackage{amsthm,amsmath,amsfonts,amssymb}
\usepackage[numbers,sort&compress]{natbib}
\usepackage[colorlinks,citecolor=blue,urlcolor=blue]{hyperref}
\usepackage{graphicx}
\usepackage[british]{babel}
\usepackage[usenames,dvipsnames]{xcolor}
\usepackage{anyfontsize}
\usepackage{amssymb}
\usepackage{comment}
\usepackage{mathrsfs}
\usepackage{mathabx}
\usepackage{xfrac}
\usepackage{mathtools}
\usepackage{wrapfig}

\usepackage[expansion=true, protrusion=true]{microtype}
\usepackage[all]{xy}
\usepackage{float}
\usepackage{subcaption}
\usepackage{longtable}
\usepackage[T1]{fontenc}

\usepackage{enumitem}

\usepackage{tabularx}
\usepackage{longtable}

\definecolor{brick}{HTML}{FF0800}
\newcommand{\brick}[1]{{\color{brick} #1}}
\newcommand{\rk}[1]{\brick{#1}\normalmarginpar\marginpar{\brick{\textleaf}}}

\theoremstyle{plain}
  \newtheorem{theorem}{Theorem}[section]
  \newtheorem{lemma}[theorem]{Lemma}
  
  \newtheorem{corollary}[theorem]{Corollary}
  
  \newtheorem{fact}[theorem]{Fact}
  \newtheorem{proposition}[theorem]{Proposition}

\theoremstyle{definition}
  
  \newtheorem{definition}[theorem]{Definition}
  
  \newtheorem{example}[theorem]{Example}

  \newtheorem{application}{Application}

\theoremstyle{remark}
  \newtheorem{remark}[theorem]{Remark}

  \newtheorem{notation}[theorem]{Notation}

  \numberwithin{equation}{section}
  \DeclareMathOperator\proj{proj}

  \DeclareMathOperator\GL{GL}
  \DeclareMathOperator\Ecc{Ecc}
  \DeclareMathOperator\Grass{Grass}
  \DeclareMathOperator\Leb{Leb}
  \DeclarePairedDelimiter\ceil{\lceil}{\rceil}

\def \N {\mathbb N}

\def \R {\mathbb R}
\def \ind{1\!\!1}

\newcommand*{\e}[1]{\text{e}^{#1}}

\begin{document}

\begin{frontmatter}

\title{Interior points and Lebesgue measure of overlapping Mandelbrot percolation sets}
\runtitle{Random self-similar sets}

\begin{aug}
	\author[A]{\fnms{Vilma}~\snm{Orgov\'anyi}\ead[label=e1]{orgovanyi.vilma@gmail.com}}
	\and
	\author[A]{\fnms{K\'aroly}~\snm{Simon}\ead[label=e2]{karoly.simon51@gmail.com}}

	%%%%%%%%%%%%%%%%%%%%%%%%%%%%%%%%%%%%%%%%%%%%%%
	%% Addresses                                %%
	%%%%%%%%%%%%%%%%%%%%%%%%%%%%%%%%%%%%%%%%%%%%%%
	\address[A]{Department of Stochastics, Institute of Mathematics, Budapest University of Technology and Economics, M\H{u}egyetem rkp. 3., H-1111 Budapest, Hungary\printead[presep={,\ }]{e1,e2}}
	
	\end{aug}
	
\begin{abstract}
	We consider a special one-parameter family of d-dimensional random, homogeneous self-similar iterated function systems (IFSs) satisfying the finite type condition.
    The object of our study is the positivity of Lebesgue measure and the existence of interior points in these random sets and in particular the existence of an interesting parameter interval where the attractor has positive Lebesgue measure, but empty interior almost surely conditioned on the attractor not being empty.
    We give a sharp bound on the critical probability for the case of positivity Lebesgue measure using the theory of multitype branching processes in random environments and in some special cases on the critical probability for the existence of interior points. Using a recent result of Tom Rush, we provide a family of such random sets where there exists a parameter interval for which the corresponding attractor has a positive Lebesgue measure, but empty interior almost surely conditioned on the attractor not being empty.
\end{abstract}

	
	\end{frontmatter}

\section{Introduction}
    In this paper we examine the positivity of the Lebesgue measure and the existence of interior points in a family of random self similar sets. For simplicity, we state and prove everything in the case when our set is on the line, however the proofs with barely any modification works in higher dimensions as well, see Appendix \ref{app2} together with a two-dimensional example which can be considered as a Mandelbrot percolation with heavy overlaps. First, we consider deterministic, homogeneous self-similar sets on the line, such that the common contraction ratio is a reciprocal of a non-negative integer and every translation is rational. The corresponding iterated function systems (IFSs) always satisfy the finite type condition---we don't impose further separation condition on the IFSs. We can naturally associate a finite set of matrices $\{\mathbf{B}_0, \dots, \mathbf{B}_{L-1}\}$ (see \eqref{a90}) to such IFSs. We randomize the set analogously to the Mandelbrot percolation by choosing a probability parameter $p$ and repeating the following two steps:
    \begin{enumerate}
        \item Within each retained cylinder interval, we consider the next level cylinder intervals.
        \item Each of them are retained with probability $p$ and discarded with probability $1-p$, independently of everything.
    \end{enumerate}
    We repeat the steps ad infinitum or until the process dies out, which we call extinction. 
    Similarly to the deterministic process the random one is described by the expectation matrices $\{\mathbf{M}_0=p\cdot\mathbf{B}_0, \dots, \mathbf{M}_{L-1}=p\cdot\mathbf{B}_{L-1}\}$.

    In our main theorem (Theorem \ref{z42}) we state that under mild conditions on the expectation matrices (all of them has a positive element in each row and each column  and there exists a strictly positive product of these matrices) the (almost sure) positivity of the Lebesgue measure of such random IFSs (conditioned on non-extinction) is equivalent to the positivity of the Lyapunov exponent $\lambda$ (see Definition \ref{x84}) of the expectation matrices. Similarly, we conjecture that the existence of interior points is determined by the lower spectral radius $\widecheck{\rho}$ (see Definition \ref{q91}) corresponding to the expectation matrices. However, we could only prove this in some special cases (see Corollary \ref{x15}). In general, we proved that if $\log\widecheck{\rho}<0$ then almost surely the attractor has empty interior (see Proposition \ref{x94}). 
Although we do not know if $\log \widecheck{\rho}>0$ guarantees the existence of interior points in our self similar set, but in Proposition \ref{x98} we give a sufficiently general checkable conditions 
which imply the existence of interior points (almost surely conditioned on non-extinction).
\begin{wrapfigure}{r}{0.6\linewidth}\begin{center}
    \includegraphics[width=\linewidth]{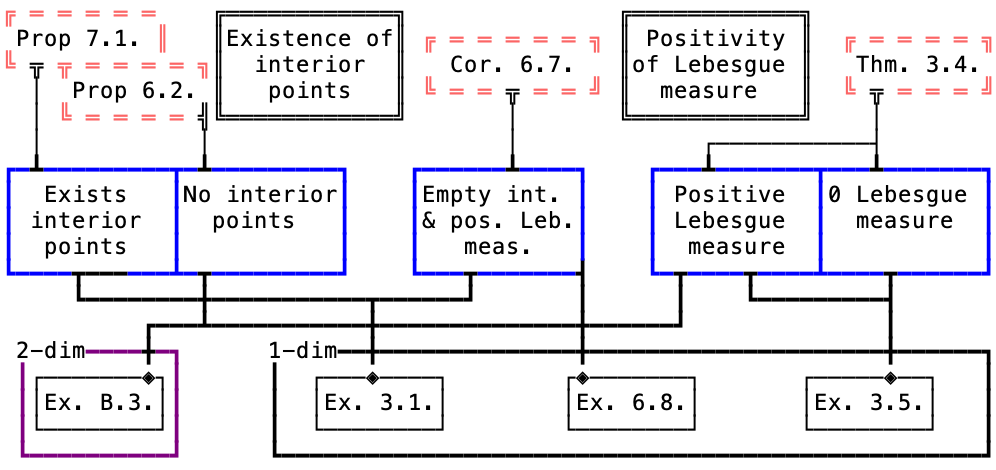}
\end{center}
  \end{wrapfigure}

    Using that the conditions $\lambda>0$ and $\log\widecheck{\rho}<0$ guarantees that the attractor has positive Lebesgue measure but empty interior almost surely (the existence of such deterministic self-similar set on the line is currently unknown), in Section \ref{x16} we investigate the possibility of the existence of a parameter interval where the above two simultaneously holds. It turns out that the existence of this parameter interval follows from the strict convexity of the negative part of the pressure function $P(q)$ (see Figure \ref{x85}), which is already known to be the case (see Theorem \ref{x90}) in the case, when the expectation matrices are invertible and the monoid generated by them is pinching and twisting (see Definition \ref{x18}). We provide a sketch of a visual table of contents summarizing the main results and examples.
    
\section{Notation}
For $k>0$, $[k]:=\{0, \dots, k-1\}$.
We denote the vectors and matrices by boldface letters; in particular $\mathbf{k}:=(k, \dots, k)$.
For two $N$-dimensional vectors $\mathbf{u}=(u_0, \dots, u_{N-1})$ and $\mathbf{v}=(v_0, \dots, v_{N-1})$ and the $N\times N$ matrices $\mathbf{U}=(u_{i,j})_{i,j\in[N]}$ and $\mathbf{V}=(u_{i,j})_{i,j\in[N]}$ let
\begin{align}
	 & \mathbf{u}\cdot \mathbf{v}:= u_0v_0+ \dots+ u_{N-1}v_{N-1} ,\,
	\mathbf{u}^{\mathbf{v}}=\prod_{i=0}^{N-1}u_{i}^{v_i}       \text{ and } \\
	 &
	\begin{array}{lr}
		\mathbf{u} \leq \mathbf{v} \\
		\mathbf{U}\leq \mathbf{V}
	\end{array}\text{ if and only if }
	\begin{array}{lr}
		u_i\leq v_i \\
		u_{i,j}\leq v_{i,j}
	\end{array}\text{ for all $i,j\in [N]$}.\label{w76}
\end{align}
We further use the strict equality version of \eqref{w76}, when all $\leq$ is replaced with $<$.

Some further notation we use throughout paper, including the place of the first occurrence:

\begin{center}
	\setlength\extrarowheight{4pt}
	\begin{longtable}{@{}m{0.07\textwidth} | m{0.78\textwidth} m{0.1\textwidth}}
	symbol& explanation&link\\
	\hline

    $\Lambda_{\mathcal{S}}$&the attractor corresponding to the IFS $\mathcal{S}$&   \eqref{a97}\\
    $L, M$&the common contraction ratio of the IFSs, $L\in\mathbb{N}$, $L\geq 2$; and $M=\#\mathcal{S}$&  \eqref{a99}\\
    $[K]$&$\{0,\dots,K-1\}$&  \\

    $\Pi(.)$&the natural projection from the symbolic space ($\Sigma_M=[M]^{\mathbb{N}}$) to the attractor&   \eqref{a96}\\

    $N$&number of basic intervals: $J^{(k)}$ ($k\in [N]$) &   \eqref{a92}\\

    $J _{\pmb{\theta } }^{(k) }$&\footnotesize $\left[b_k L+\sum_{\ell =1}^{n}\theta _\ell L^{-(\ell -1)}, b_k L+\sum_{\ell =1}^{n }\theta _\ell L^{-(\ell -1)}+L^{-(n-1)}\right]$,  $\pmb{\theta}\in[L]^n,\,k\in [N]$&   \eqref{a91}\\
						
	$\mathbf{B}_{\theta}$ & the $N\times N$ coding matrix &\eqref{a90}	\\
					
    $\Lambda_{\mathcal{S}, p}$&the random attractor corresponding to the IFS $\mathcal{S}$ and probability $p$&    \eqref{a95}\\

	$\lambda$&Lyapunov exponent &  Def. \ref{x84}\\
									
    $\mathbf{M}_{\theta}$&expectation matrix in the environment $\theta$ and also for the CISSIFS& \eqref{x25}\\																						  
	$q^{(k)}(\overline{\pmb{\theta}})$, $\mathbf{q}(\overline{\pmb{\theta}})$&the probability that the process starting with one individual of type-$k$ becomes extinct, and the vector of these, $(q^{(k)}(\overline{\pmb{\theta}}))_{k\in[N]}$ resp.& \eqref{x59}\\

    $\widecheck{\rho}(\mathcal{B})$& Lower spectral radius corr. to the set of matrices $\mathcal{B}=\{\mathbf{B}_0, \dots, \mathbf{B}_{L-1}\}$&
    Def \ref{q91}						          
    \end{longtable}                                        
	\end{center}
\section{Coin-tossing integer self-similar IFSs on the line}\label{z22}
In this chapter we first formally introduce a special type of homogeneous deterministic IFSs (called \textit{integer self-similar IFSs on the line}, or shortly ISSIFS). This  serves as a skeleton for the random IFS (called \textit{coin tossing} ISSIFSs, CISSIFS), which is the object of our interest throughout the paper. It is defined in the second part of this chapter. Then we state the main theorem (Theorem \ref{z42}) of the paper, which describes the positivity of the Lebesgue measure of a subfamily of CISSIFSs in terms of the probability parameter $p$ and the Lyapunov exponents corresponding to the matrices given in \eqref{a90}.
Such ISSIFSs were considered by Ruiz in \cite{ruiz2009dimension} and more general schemes for randomization occurs in \cite{falconer2014exact}, \cite{zbMATH06808299}.
\subsection{Deterministic integer self-similar IFS on the line}\label{z45}

Self-similar IFSs on the line are finite list of contracting similarities of $\mathbb{R}$ which can be presented as $\mathcal{S}=\left\{ S_i(x):=r_ix+t_i \right\}_{i=0}^{M-1}$ for some $r_i\in(-1,1)\setminus \{0\}$. In this paper we confine ourselves to the special case when
\begin{enumerate}
    [label=\textbf{(\alph*)}]
    \item All contractions are the reciprocal of the same integer: That is $r_i=\frac{1}{L}$ for an $L\geq 2$, $L\in\mathbb{N}$.
    \item All translations $t_i$ are rational numbers: That is $t_i\in \mathbb{Q}$.
\end{enumerate}
If we multiply all the translation parameters with the same positive number, the IFS obtained, has the same properties as the original one. Hence, without loss of generality we may assume, and throughout the paper we will assume that $\mathcal{S}$ is of the form
\begin{multline}\label{a99}
    \mathcal{S}:=\left\{
        S_i(x):=\frac{1}{L}x+t_i
     \right\}_{i=0}^{M-1}
     , S_i:\mathbb{R}\to\mathbb{R},
    \\
    L\in\N \setminus \left\{0,1\right\},
    \, t_i\in \mathbb{N}, 0=t_0\leq \cdots\leq t_{M-1},\
    L-1|t_{M-1}.
    \end{multline}
    We call an Iterated Function System (IFS) in this form \textit{Integer Self-Similar IFS} (ISSIFS).
 The attractor of $\mathcal{F}$ is the unique non-empty compact set $\Lambda $ for which $\Lambda_{\mathcal{S}}=\Lambda =\bigcup _{i=1}^{M}S_i(\Lambda )$.
It is easy to see that
\begin{equation}
\label{a97}
\Lambda =\bigcap _{n=1}^{\infty  }
\bigcup _{i_1\dots  i_n}S_{i_1\dots  i_n}(I),
\end{equation}
where $S_{i_1\dots  i_n}:=S_{i_1}\circ\cdots\circ S_{i_n} $ and $I:=[0,\text{Fix}(S_{M-1})]$, where $\text{Fix}(S_{M-1})$ is the fixed point of $S_{M-1}$.
The corresponding symbolic space is
$\Sigma_M :=\left\{ 0,\dots  M-1 \right\}^{\mathbb{N}}$. The natural coding of the points of $\Lambda $
by the elements of $\Sigma_M $ is given by
\begin{equation}
\label{a96}
\Pi (\mathbf{i}):=\lim\limits_{n\to\infty}
S_{i_1\dots  i_n}(0), \quad \mathbf{i}=(i_1,i_2,\dots  ).
\end{equation}
Let $\mu :=\left(\frac{1}{M} ,\dots  ,\frac{1}{M} \right)^{\mathbb{N}}$ be the uniformly distributed measure on $\Sigma_M $ and denote its push forward measure by
\begin{equation}
\label{a94}
\eta :=\Pi _*\mu.
\end{equation}

\begin{figure}
    \centering
    \begin{subfigure}[b]{0.65\textwidth}
        \centering
        \includegraphics[height=4.2cm]{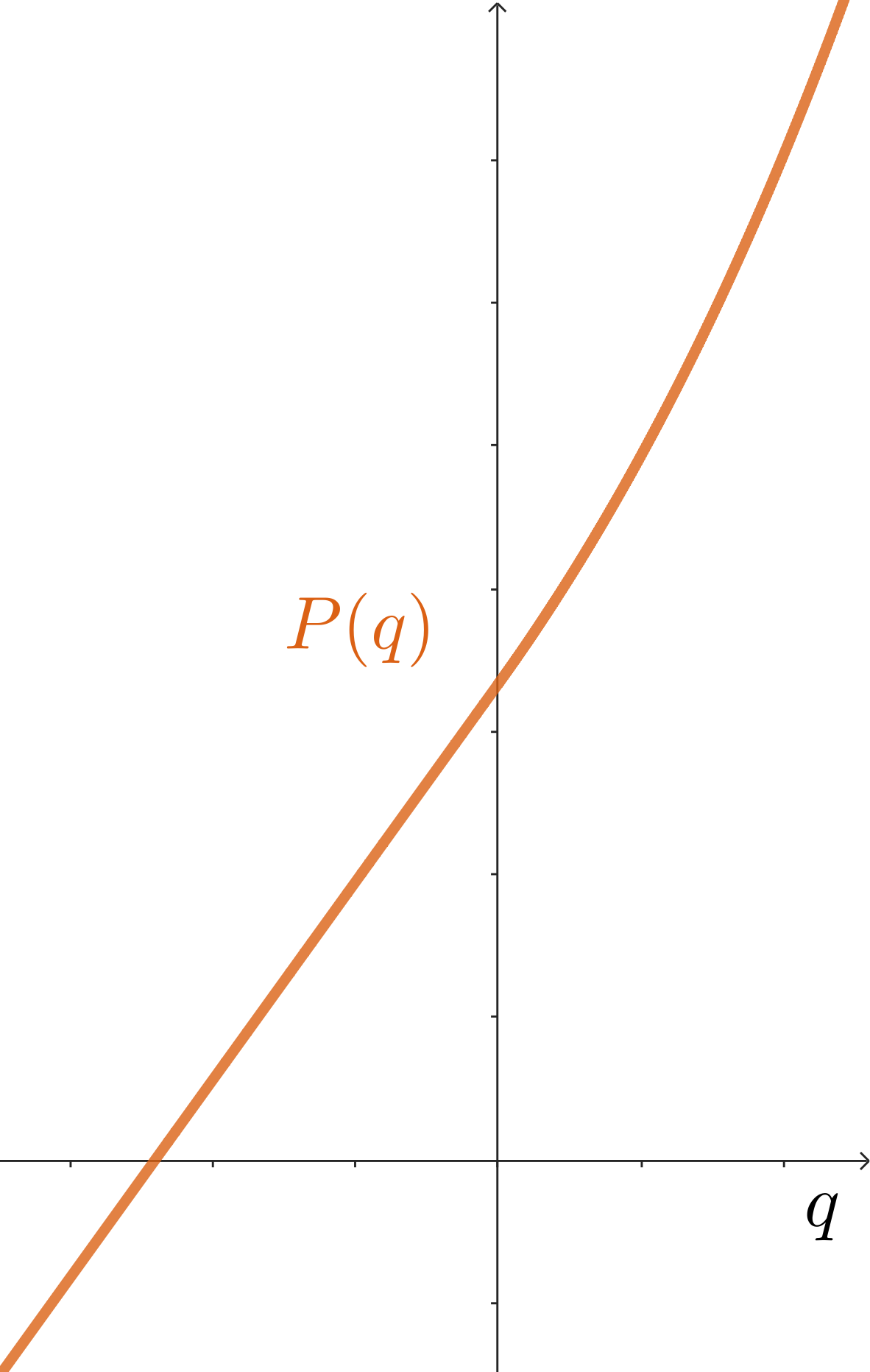}
        \includegraphics[height=4.2cm]{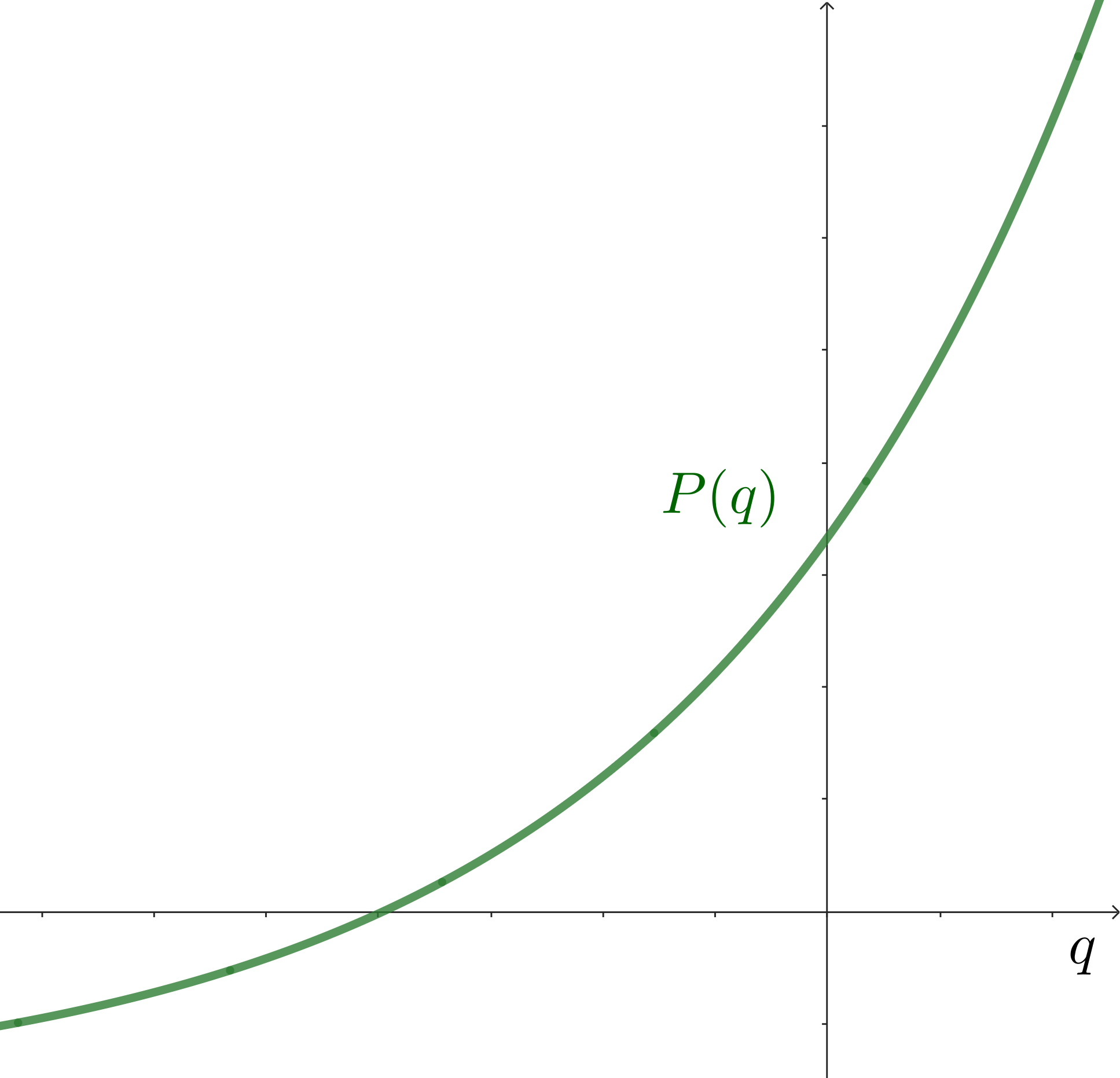}
        \caption{The shape of the pressure function. In the second case we have the \textit{interesting parameter interval} exists.}
        \label{x85}
    \end{subfigure}
    \hfill
    \begin{subfigure}[b]{0.34\textwidth}
        \centering
        \includegraphics[height=4.4cm]{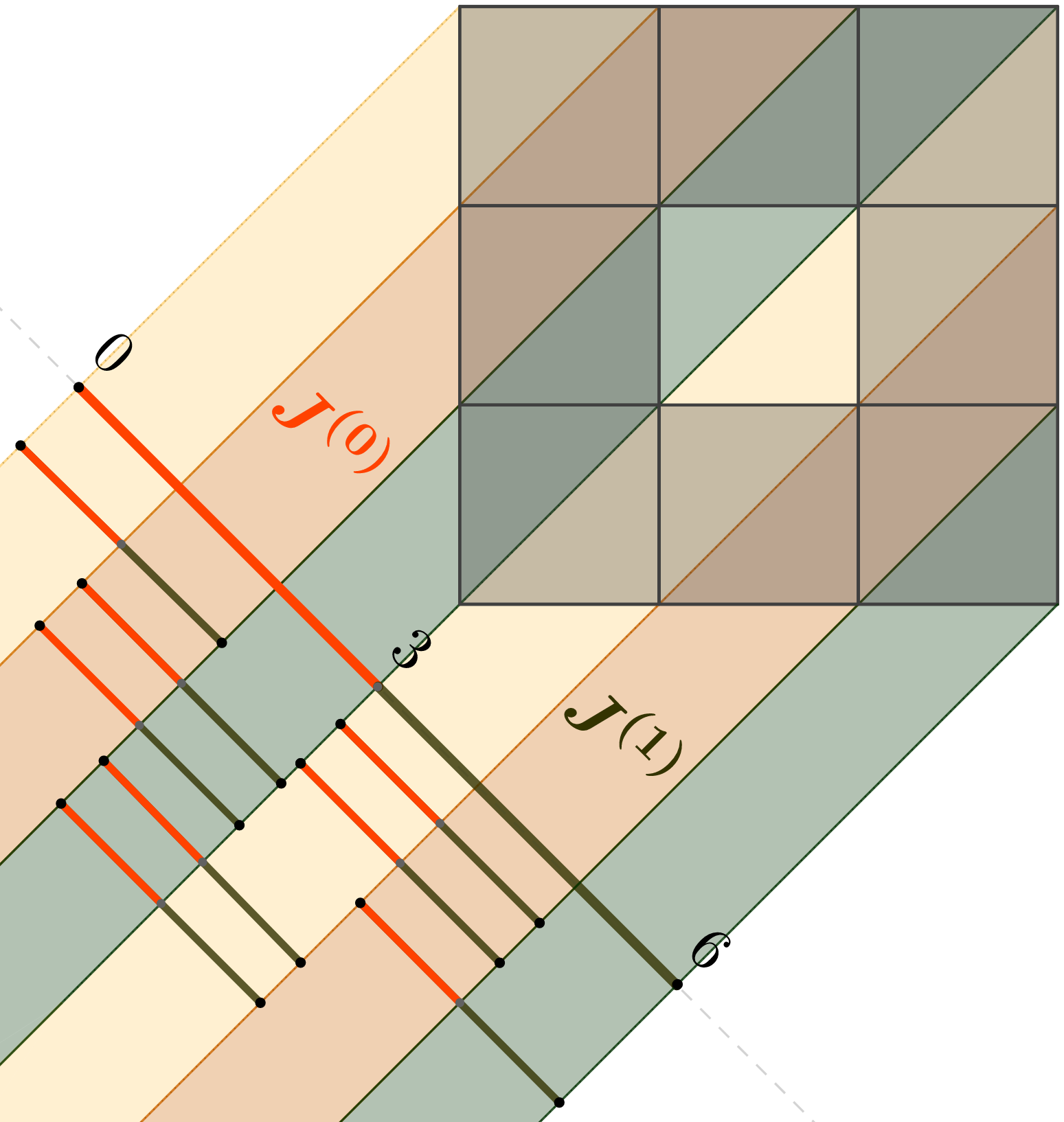}
        \caption{45-degree projection of the Sierpi\'nski carpet.}
        \label{x86}
    \end{subfigure}
    \caption{}
    \label{x87}
\end{figure}

\subsubsection{Setup through an example}\label{x28}
Before treating a general set of this type, we examine  the example of the $45$-degree projection of the Sierpi\'nski carpet. We recommend looking at Figure \ref{x86} while reading the description.
\begin{equation}
    \mathcal{S}=\{S_i(x)=\frac{1}{3}x+t_i\}_{i=0}^{7},
\end{equation}
where $t_0=0$, $t_1=t_2=1$, $t_3=t_4=2$, $t_5=t_6=3$ and $t_7=4$. In this case the contraction ratio is the reciprocal of $L=3$, the number of functions are $M=8$.
Consider the triadic intervals 
\begin{equation}
    \mathcal{D}_k:= 
        \left\{\left[ (i-1)3^{-k},i 3^ {-k} \right]:i\in \mathbb{Z}\right\},\quad k\in\left\{ -1,0,1,2,\dots \right\},
\end{equation}
We are particularly interested in the intervals $J^{(0)}=[0,3], J^{(1)}=[3,6]\in \mathcal{D}_{-1}$, which we call \textit{basic intervals}. 

In particular, since the images of the basic intervals under the iterates of the functions of our IFS are triadic subsets of the basic intervals, we keep track of the number of cylinders of intersecting a triadic subinterval of $J^{(k)}$.

Namely, consider the triadic subsets,
$$
J_{\pmb{\theta}}^{(k)}=
\left[3 k+\sum_{\ell =1}^{n}\theta _\ell 3^{-(\ell -1)},3 k+\sum_{\ell =1}^{n }\theta _\ell3^{-(\ell -1)}+3^{-(n-1)}\right],
$$  
of $J^{(k)}$ for $\pmb{\theta}\in[3]^n$.
For any $\mathbf{i}\in[M]^{n}$ and basic interval $J^{(\ell)}$ for one of the basic intervals $J^{(k)}$ and $\pmb{\theta}\in [3]^n$ we have $S_{\mathbf{i}}(J^{(\ell)})=J_{\pmb{\theta}}^{(k)}$. Following Ruiz (\cite{ruiz2009dimension}) we define altogether $L=3$, $2\times 2$ ($2$ is the number of basic intervals, which determines the shape) matrices
\begin{equation}\label{a90}
    \mathbf{B}_{\theta} (i,k) :=
    \#\left\{ \ell \in [M]:
    S_\ell (J^{(k)})=J _{\theta }^{(i) }
    \right\}, 
\end{equation}
for $\theta\in[3],\,i,k\in[2]$.
The step by step construction of $\mathbf{B}_{1}$ is as follows. The first row describes the cylinders intersecting $J^{(0)}_{1}$. One can verify that $S_0(J^{(1)})=S_1(J^{(0)})=S_1(J^{(0)})=J^{(0)}_{1}$, meaning that $B_{1}(0,0)=2$, $B_{1}(0,1)=1$. This can be also read from the figure by inspecting the middle, light orange stripe through $J^{(0)}$. This contains two  orange (image of $J^{(0)}$) and one dark grey (image of $J^{(1)}$) interval.
Altogether 
\begin{equation*}
    \mathbf{B}_0=
\begin{bmatrix}
1 & 0 \\
2 & 2  \\

\end{bmatrix}, \quad
\mathbf{B}_1=
\begin{bmatrix}
2 & 1 \\
1 & 2  \\

\end{bmatrix}, \quad
\mathbf{B}_2=
\begin{bmatrix}
2 & 2 \\
0 & 1  \\
\end{bmatrix}.
\end{equation*}
\subsubsection{General setup}
Recall that $L\geq 2$ is the reciprocal of the contraction ratio of our IFS.
We introduce the family of partitions (mod $0$)
of $\mathbb{R}$ and refer to the elements of these partitions as $L$-adic intervals.
\begin{equation}
\label{a93}
\mathcal{D}_k:=
\left\{
    \left[ (i-1)L^{-k},iL^{-k} \right]:
    i\in \mathbb{Z}
 \right\},\quad
 k\in\left\{ -1,0,1,2,\dots   \right\}.
\end{equation}
Particular attention is given to those
elements of $\mathcal{D}_{-1}$ which have positive $\eta$-measure ($\eta$, the push-forward of the uniform measure on the codespace was defined in \eqref{a94}). We call these intervals \textit{basic intervals}. The number of basic intervals is finite, say $N$,
and the basic intervals are denoted by
\begin{equation}
\label{a92}
\left\{J^{(U)}:=[b_U L,(b_U +1)L]  \right\}_{U=0}^{N-1},
\quad
b_kóU\in \mathbb{N},\
b_U<b_{U+1},\
U=0,\dots  ,N-2.
\end{equation}

The smallest interval that contains all the basic intervals is $I=[0,\text{Fix}(S_{M-1})]=\left[ 0,L\frac{t_{M-1}}{L-1} \right] $.
Here we used that by assumption $\frac{t_{M-1}}{L-1}\in\mathbb{N}$. For every $U\in [N]$
the basic interval $J^{(U)}$ subdivides
into $L^n$ congruent subintervals contained in
$\mathcal{D}_{n-1}$ (of length $L^{-(n-1)}$)
which are denoted by $J _{\pmb{\theta}} ^{(U) }$,
where $\pmb{\theta }=(\theta _1,\dots  ,\theta _n)\in [L]^n$. More precisely,
for $U\in[N]$ and $\pmb{\theta}  \in[L]^n$:
\begin{equation}
\label{a91}
J _{\pmb{\theta } }^{(U) }=
\left[
b_U L+\sum_{\ell =1}^{n }
\theta _\ell L^{-(\ell -1)}  ,
b_U L+\sum_{\ell =1}^{n }
\theta _\ell L^{-(\ell -1)}
+
L^{-(n-1)}
 \right].
\end{equation}
For every $\theta \in[L]$ we define the $N\times N$
matrix exactly as in \ref{a90}.
Then one can easily check  that
for a $\pmb{\theta }=(\theta _1,\dots  ,\theta _n)\in [L]^n$
we have
\begin{equation}
\label{z44}
\mathbf{B}_{\pmb{\theta }}(U ,V)=\left( \mathbf{B}_{\theta _1}\cdots \mathbf{B}_{\theta _n} \right) (U,V)
= 
\#
\left\{
(\ell _1,\dots  ,\ell_n)\in [M]^n :
S_{\ell _1\dots  \ell_n}(J^{(V)})=J _{ \pmb{\theta }}^{(U) }
 \right\}.
\end{equation}

\subsection{Formal definition of CISSIFS}\label{a77}
First we define the randomly labelled $M$-ary trees. We fix a $p\in(0,1]$ and an integer $M\geq 2$.
 The $M$-ary tree $\mathcal{T}_M$ is defined as follows: the root is denoted by $\emptyset $ and all other nodes are the strings over the alphabet $[M]$.
The level of a node is its length as a string. The $n$-th
level of the tree $\mathcal{M}_n$ is the set of all nodes of level $n$. The offspring of a node $i_1\dots  i_n\in\mathcal{M}_n$
are the nodes $i_1\dots  i_ni_{n+1}$ for all $i_{n+1}\in[M]$.

Let $\Omega_M :=\left\{ 0,1 \right\}^{\mathcal{T}_M}$ be the set of labelled trees (with labels chosen from $\left\{ 0,1 \right\}$).
The standard $\sigma $-algebra on $\Omega_M $ is denoted by $\mathcal{A}$. We introduce  the probability measure $\mathbb{P}_p$ on $(\Omega_M ,\mathcal{A})$
by assigning a Bernoulli random  variable  $X_{i_1\dots  i_n}$ to each node $i_1\dots  i_n\in\mathcal{T}_M$. The
probability measure $\mathbb{P}_p$ is defined such that
\begin{enumerate}[label=\textbf{(\alph*)}]
    \item $X_{\emptyset }\equiv 1$,
    \item $\mathbb{P}_p(  X_{i_1\dots  i_n}=1)=p$ for all $n\geq 1$ and $i_1\dots  i_n\in\mathcal{T}_M$,
    \item $\left\{ X_{i_1\dots  i_n} \right\}_{i_1\dots  i_n\in \mathcal{T}_M}$ are independent.
\end{enumerate}
 For an $\pmb{\omega}_M \in \Omega $ and a $n\geq 1$ we write
 \begin{equation}
 \label{a98}
 \mathcal{E}_n(\pmb{\omega} ):=\left\{
    i_1\dots  i_n\in\mathcal{T}_M: X_{i_1}=X_{i_1i_2}=\cdots X_{i_1i_2\dots  i_n}=1
  \right\}.
 \end{equation}
 Moreover, let
\begin{equation}
\label{z35}
\mathcal{E}_{\infty  }(\pmb{\omega} ):=
\left\{
    \mathbf{i}\in\Sigma _M:
    X_{\mathbf{i}|_n}=1,\ \forall n \geq 1
\right\}.
\end{equation}

 The event that $\mathcal{E}_\infty \ne \emptyset $
 is called \textit{non-extinction}.
 Then the \textit{coin tossing integer self-similar set corresponding to the ISSIFS (see \eqref{a99}) $\mathcal{S}$ and the probability $p$} is
 \begin{equation}
 \label{a95}
 \Lambda_{\mathcal{S},p}(\pmb{\omega} )=
 \Pi (\mathcal{E}_{\infty  }(\pmb{\omega} )),
 \end{equation}
 where $\Pi $ was defined in \eqref{a96}.
 Assume that   the deterministic self-similar IFS  $\mathcal{S}$ is of the form of \eqref{a99}, and  satisfies the Open Set Condition (see
    \cite{falconer2004fractal}).
It follows from a theorem of
Falconer \cite{zbMATH04011530} and Mauldin-Williams \cite{mauldin1986random}
that for $\mathbb{P}_p$-a.e. $\pmb{\omega} \in\Omega $
\begin{equation}
    \label{u65}
    \dim_{\rm H}  \Lambda_{\mathcal{S},p}(\pmb{\omega} )
    =\dim_{\rm B}  \Lambda_{\mathcal{S},p}( \pmb{\omega} )=
     \frac{\log (Mp)}{\log L}
    ,\quad \text{if}\quad
    \Lambda_{\mathcal{S},p}(\pmb{\omega} )\ne \emptyset.
    \end{equation}
\subsection{Random sponges of $\mathbb{R}^d$ and their projections} \label{z14}
Analogously to CISSIFS we can define \textit{random $d$-dimensional sponges} ($d\geq 2$), with the modification that instead of the (deterministic) ISSIFS $\mathcal{S}$ we consider an IFS $\mathcal{F}$ corresponding to a $d$-dimensional sponge. Namely, choose a parameter $K$ and an arbitrary subset $\left\{\mathbf{d}_i\right\}_{i=0}^{M-1}$ (of size $2 \leq M\leq K^d-1$) of the set $\left\{\mathbf{t}_i\right\}_{i=0}^{K^d-1}$, the enumeration of the left bottom corners of the $K$-mesh cubes contained in $[0,1]^d$.
The IFS $\mathcal{F}$ corresponding to $K$ and $\left\{\mathbf{d}_i\right\}_{i=0}^{M-1}$ is
$$\mathcal{F}:=\left\{f_i(\mathbf{x})=
\frac{\mathbf{x}}{K}+\mathbf{t}_i\right\}_{i=0}^{M-1}.$$

For $\mathbf{a}\in\mathbb{R}^d$ we define the projection $\proj_{\mathbf{a}}:\mathbb{R}^d\to \mathbb{R}$, $\proj_{\mathbf{a}}(\mathbf{x})=\mathbf{a}\cdot \mathbf{x}$, where $\cdot$ denotes the standard dot product in $\mathbb{R}^d$. We say that $\proj_{\mathbf{a}}$ is a \textit{rational projection} if all coordinates of $\mathbf{a}$ is a rational number. The rational projections of random $d$-dimensional sponges to lines are examples of CISSIFSs.
We repeatedly refer to the following example.
\begin{figure}
    \centering
    \includegraphics[width=0.45\linewidth]{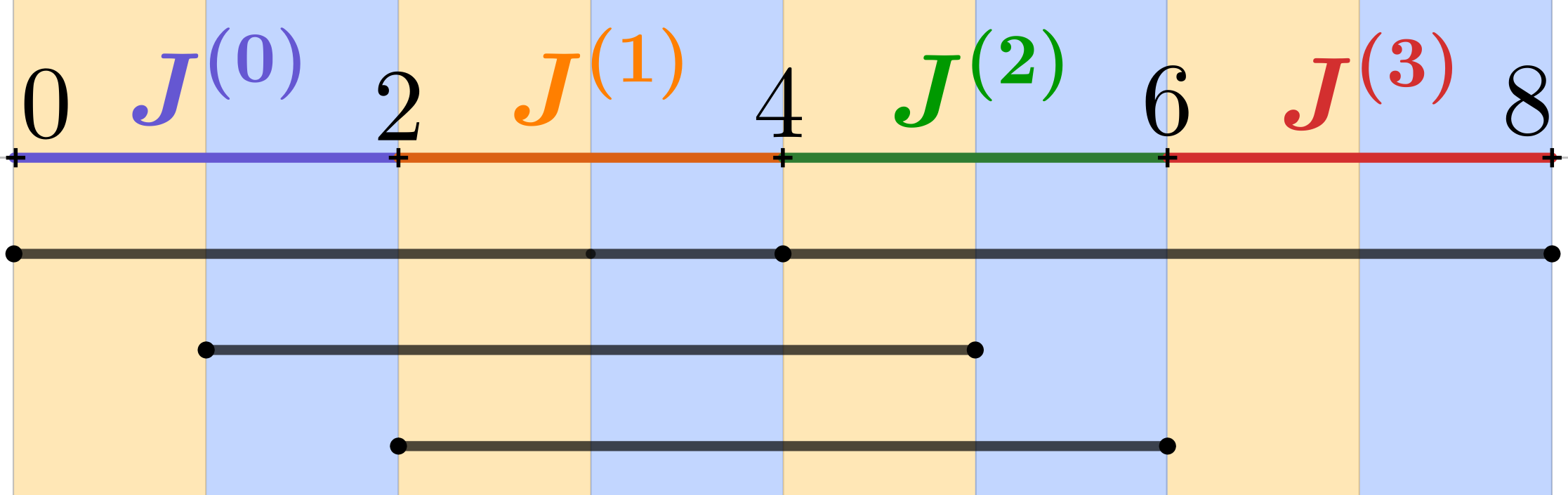}
    \caption{Example \ref{x20}.}\label{a71}
  \end{figure}
\begin{comment}
\begin{example}[$0-1-3$ problem, with contraction ratio $1/2$]\label{x57}
    Let $$S_i(x)=\frac{1}{2}x+t_i,$$
    where $t_i=0,1,3$. Note that this IFS is the projection of the previously examined right-angled Sierpi\'nski carpet along the lines of slope $2$. Consider the CISSIFS corresponding to this IFS. In this case $N=3$, $L=2$, the expectation matrices are 
    \begin{equation*}
        \mathbf{M}_0=p\cdot
    \begin{bmatrix}
    1 & 0 & 0 \\
    0&1&1\\
    0&1&0\\
    \end{bmatrix}, \quad
    \mathbf{M}_1=p\cdot
    \begin{bmatrix}
        1 & 1 & 0 \\
        1&0&1\\
        0&0&1\\
    \end{bmatrix}.
    \end{equation*}
    See Example \ref{x58}, \ref{x56}, and Application \ref{x22}. 
\end{example}
\end{comment}
\begin{example}\label{x20}
    Let $$S_i(x)=\frac{1}{2}x+t_i,$$
    where $t_i=0,1,2,4$. Note that this IFS is not a projection of a 2 dimensional carpet. Consider the CISSIFS corresponding to this IFS. In this case $N=4$, $L=2$, the expectation matrices are 
    \begin{equation*}
        \mathbf{M}_0=p\cdot
    \begin{bmatrix}
1& 0& 0& 0\\
1& 1& 1& 0\\
1& 0& 1& 1\\
0& 0& 1& 0\\

    \end{bmatrix}, \quad
    \mathbf{M}_1=p\cdot
    \begin{bmatrix}
        1& 1& 0& 0\\
        0& 1& 1& 1\\
        0& 1& 0& 2\\
        0& 0& 0& 1\\

    \end{bmatrix}.
    \end{equation*}
    See Example \ref{x24}, and Application \ref{x22}.
\end{example}

\subsection{The result}\label{z43}
Let $\mathcal{S}=\left\{ S_i \right\}_{i=0}^{M-1 }$ be an ISSIFS of the form of \eqref{a99}. Let $p\in[0,1]$. We write $\Lambda _{\mathcal{S},p}$ for the  corresponding coin tossing integer self-similar set defined in Section \ref{a77}.
Let $\mathcal{B}:=\left\{ \mathbf{B}_0,\dots  ,\mathbf{B}_{L-1} \right\}$, where the matrices $\mathbf{B}_\theta $, $\theta \in[L]$  were defined in \eqref{a90}. Let $\nu$ be any ergodic measure on $\Sigma$. 
\begin{definition}[Good set of matrices]\label{x52}
In the above setup we say that $\mathcal{B}$ is good (w.r.t the ergodic measure $\nu$) if 
\begin{itemize}
    \item every element of $\mathcal{B}$ is non-negative;
    \item every element of $\mathcal{B}$ is allowable (i.e. every row and column contains a strictly positive element), and
    \item there exists $(\theta_1,\dots, \theta_n)$ such that $\nu(\{\overline{\pmb{\theta}}:\overline{\pmb{\theta}}|_n=(\theta_1,\dots, \theta_n)\})>0$ the product $\mathbf{B}_{\theta_1}\cdots \mathbf{B}_{\theta_n}$ is strictly positive.
\end{itemize}
\end{definition}
We inspect the push-forward of the measure $\nu$ by the coding map in each of the basic intervals. Namely, for each $U\in [N]$ we define the projection $\Gamma^{(U)}:\Sigma\to J^{(U)}$,
\begin{equation}
\Gamma^{(U)}(\overline{\pmb{\theta}}):= b_UL+\sum_{\ell =1}^{\infty }
\theta _\ell L^{-(\ell -1)},
\end{equation}
recall that $b_UL$ is the left-endpoint of the interval $J^{(U)}$. This mapping is a ($\nu$-mod 0) bijection between the code-space $\Sigma$ and $J^{(U)}$. We define the (not probability) measure  $\widetilde{\nu}$ on the set $\bigcup_{U\in [N]} J^{(U)}$ in the following way:
\begin{equation}\label{a69}
\widetilde{\nu}^{(U)}:=\Gamma^{(U)}_*\nu\quad
\text{ and }\quad
    \widetilde{\nu}=\sum_{U\in [N]} \widetilde{\nu}^{(U)}.
\end{equation}

The Lyapunov exponent corresponding to the ergodic measure $\nu$ and the good matrices $\mathcal{B}$:
\begin{definition}[Lyapunov exponent]\label{x84} We are given an
	ergodic measure $\nu $ on $(\Sigma ,\sigma )$ and a
	$\mathcal{B}=\{\mathbf{B}_{i}\}_{i\in\mathcal{I}}$, which is good
	with respect to $\nu $.
	The \textit{Lyapunov}-\textit{exponent} corresponding to $\mathcal{B}$ and $\nu$ is
	\begin{equation*}
		\lambda:=\lambda(\nu, \mathcal{B})= \lim_{n \to \infty} \frac{1}{n} \log\|\mathbf{B}_{\overline{\pmb{\theta}}|_n}\|_{*} \; \text{for $\nu$-almost every $\overline{\pmb{\theta}}\in \Sigma$},
	\end{equation*}
    where $\|.\|_{*}$ is any submultiplicative matrix norm.
\end{definition}
The existence of $\lambda$ as defined above follows from
\cite[Corollary 10.1.1]{walters2000introduction}, and the ergodicity of $\nu$.

We state the main theorem of the paper.
\begin{theorem}\label{z42}
We suppose that $\mathcal{B}$ is good w.r.t. the ergodic measure $\nu$. Consider the Lyapunov exponent $\lambda(\nu,\mathcal{B})$.
\begin{enumerate}
\item If $\e{-\lambda }<p$, then $\widetilde{\nu}(\Lambda _{\mathcal{S},p})>0$ almost surely conditioned on non-extinction.
\item If $\e{-\lambda }\geq p$ then $\widetilde{\nu}(\Lambda _{\mathcal{S},p})=0$ almost surely.
\end{enumerate}

In particular, when $\nu=\left(\frac{1}{L}, \dots, \frac{1}{L}\right)^{\N}$, then $\widetilde{\nu}$ is the (positive) constant multiple of the Lebesgue measure restricted to the union of the basic intervals: $\bigcup_{U\in[N]} J^{(U)}$. In this way, since $\Lambda_{\mathcal{S},p}\subset \bigcup_{U\in[N]} J^{(U)}$ the statement above is an if and only statement about the positivity of Lebesgue measure of the random attractor.
\end{theorem}
We prove this theorem in Section \ref{x81} with the use of Theorem \ref{z68} stated below.
\begin{example}[Random right-angled Sierpi\'nski gasket] \label{u99}
    In this example we look at the random one-dimensional system which is the $45$-degree projection of the random right-angled Sierpi\'nski gasket, the attractor (see Figure \ref{a70}) of the following self-similar IFS in $\R^2$.
    \begin{equation*}%\label{u48}
    \mathcal{S}:
    =\left\{S_{i}(\mathbf{x})=\frac{1}{2}\mathbf{x}+\mathbf{t}_i\right\}_{i=0}^{3},
    \end{equation*}
    where $\left\{\mathbf{t}_i\right\}_{i=0}^{3}$ is an enumeration of the set
  $
    \left\{0,\frac{1}{2}\right\}^2 \setminus \left\{\left(\frac{1}{2}, \frac{1}{2}\right)\right\}.
$
    We obtain the randomized set (with parameter $p$) as described in Section \ref{z14} for the deterministic IFS above which we denote by $\mathcal{G}_p$. Consider $\proj: \R^2\to\R$,
        $\proj(x,y):=-x+y$,
    the rescaled version of the $45$-degree projection. We investigate $\proj(\mathcal{G}_p)$, see Figure \ref{a70}. In this case  $L=N=2$ and
    $\mathcal{S}=\left\{ S_i(x)=\frac{1}{2}x+2(i-1) \right\}_{i=1}^3 $. The types are determined by the basic intervals,
    $J^{(0)}:=\left[ 0,2 \right]$ and $J^{(1)}:=[2,4]$.
    The environments are identified with
     $\overline{\pmb{\theta}}=(\theta_1,\theta_2,\dots   ) \in\{0,1\}^{\mathbb{N}}$.
    \begin{figure}
        \centering
        \includegraphics[width=0.7\linewidth]{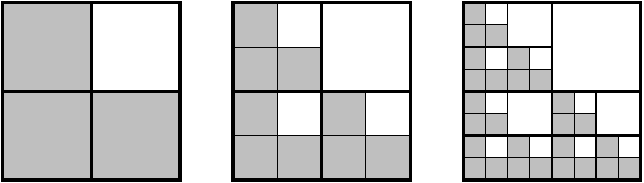}\includegraphics[width=0.3\linewidth]{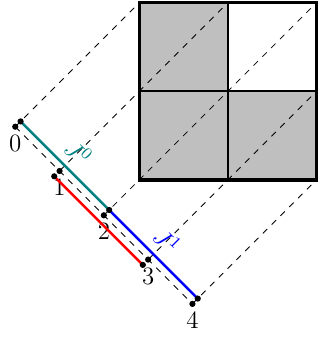}
        \caption{The approximations of the (deterministic) right-angled Sierpi\'nski gasket and the IFS corresponding to $\proj(\mathcal{G}_p)$.}\label{a70}
      \end{figure}
By \eqref{a90} we have
    \begin{equation*}
        \mathbf{B}_0=
        \begin{bmatrix}
        1 & 0 \\
        1 & 1  \\

        \end{bmatrix},
        \mathbf{B}_1=
        \begin{bmatrix}
        1 & 1 \\
        0 & 1  \\

        \end{bmatrix}.
        \end{equation*}
According to \cite{pollicott_Vytnova} the lower and upper bounds for the Lyapunov exponent $\lambda$ in this case is $0.3961$ and $0.3962$ respectively. Meaning that if we choose $p>\e{-\lambda}\geq 0.6729\dots$ then $\Lambda_{\mathcal{G}, p}$ has positive Lebesgue measure almost surely conditioned on non-extinction. Note that in \cite[Section 2.3]{OurPaper} we also gave a lower bound, but in that case the lower bound for $\Leb(\mathcal{G}_p)>0$ was much worse, $p>\frac{1}{\sqrt{2}}=0.707\dots$.
\end{example}

\section{Extinction of MBPREs}\label{x23}
In this paper we only consider a special setup for Multitype branching processes in random environments (MBPREs), for the general construction see for example \cite{bp_renv}, \cite{kaplan_branching}, \cite{zbMATH03738673}, \cite{vatutin2021multitype}. 

From now on we refer to distributions with their probability generating functions (pgfs). 
In our special case we are given a finite alphabet $[L]=\{0, 1,\dots,L-1 \}$ indexing the set of offspring distributions,
$$
\mathbf{F}=(\mathbf{f}_0,\dots ,\mathbf{f}_{L-1}),\quad \mathbf{f}_{\theta}=(f_{\theta}^{(0)},\dots,f_{\theta}^{(N-1)}),
$$
where for each $\theta\in[L]$ $j\in[N]$, $f_{\theta}^{(j)}$ is an $N$-dimensional probability distribution taking values from $\mathbb{N}_0^N$, the space of non-negative integer vectors with $N$ components.

Namely,
$$
f_{\theta}^{(j)}(\mathbf{s})=\sum_{\mathbf{z}\in \mathbb{N}_{0}^N}f_{\theta}^{(j)}[\mathbf{z}]\mathbf{s}^\mathbf{z},
\text{ for }
\mathbf{s}\in[0,1]^N,
$$
where for $(z_0,\dots, z_{N-1})=\mathbf{z}\in \mathbb{N}_0^N$,  $f_{\theta}^{(j)}[\mathbf{z}]$ denotes the probability that the value of the (vector) random variable distributed according to $f_{\theta}^{(j)}$ is $\mathbf{z}$.

We define $\Sigma:=[L]^{\N}$, all infinite words over our alphabet and the usual $\sigma$-algebra $\mathcal{A}$ on it. We consider the dynamical system $(\Sigma,\mathcal{A}, \sigma)$, where $\sigma$ is the left-shift map, namely for $\overline{\pmb{\theta}}=(\theta_1, \theta_2, \dots)\in \Sigma$, we have $\sigma(\overline{\pmb{\theta}})=(\theta_2, \theta_3,\dots)$. 
Assume that we are given $\overline{\theta}=\left(\theta_n\right)_{n\geq 1}  \in \Sigma$. We identify the infinite sequences $(\mathbf{f}_{\theta_1},\mathbf{f}_{\theta_2}, \dots)$, ($\mathbf{f}_{\theta_i}\in\mathbf{F}$) of probability distributions with their indices, namely $(\theta_1, \theta_2,\dots)=\overline{\pmb{\theta}}\sim(\mathbf{f}_{\theta_1},\mathbf{f}_{\theta_2}, \dots)$. 

With this identification, $\Sigma$ is the set of possible environments: For an environment $\overline{\pmb{\theta}}\in \Sigma$ let $(\mathbf{Z}_n(\pmb{\overline{\theta}}))_{n\geq 1}$ be the time-inhomogeneous branching process driven by the sequence of offspring distributions $(\mathbf{f}_{\theta_1},\mathbf{f}_{\theta_2}, \dots)$. Now we choose an ergodic (probability) measure $\nu$ on $(\Sigma,\mathcal{A}, \sigma)$---a distribution on the environments. 
The MBPRE corresponding to $\Sigma$ and $\nu$ is denoted by $\mathcal{Z}=(\mathbf{Z}_n)_{n\geq 1}$.

In the next chapter we will describe MBPREs associated to CISSIFs. For a short introduction about the construction of MBPREs see Appendix \ref{app1}.

For each $\theta\in[L]$ we define the expectation matrices 

\begin{equation}\label{x25}
	\mathbf{M}_{\theta}(i,j):=\frac{\partial f_{\theta}^{(i)}}{\partial s_j}(\mathbf{1}).
\end{equation}

 For $\overline{\pmb{\theta}}\in \Sigma$ let $q^{(k)}(\overline{\pmb{\theta}})$ denote the probability that the process $(\mathbf{Z}_n(\pmb{\overline{\theta}}))_{n\geq 1}$ starting with one individual of type-$k$ becomes extinct. We denote
 \begin{equation}\label{x59}
     \mathbf{q}(\overline{\pmb{\theta}}):=(q^{(0)}(\overline{\pmb{\theta}}), \dots, q^{(N-1)}(\overline{\pmb{\theta}})).
 \end{equation}

The following theorem regrading the extinction probability of an MBPRE described above is a straightforward consequence of \cite[Theorem 1.1]{2024MBPRE}.

Note that in the non-critical case, under the assumption of the following theorem there are only two possibilities; either 
\begin{itemize}
    \item $\mathbf{q}(\overline{\pmb{\theta}})=\mathbf{1}$ for $\nu$ almost every $\overline{\pmb{\theta}}$, or 
    \item $\mathbf{q}(\overline{\pmb{\theta}})<\mathbf{1}$ for $\nu$ almost every $\overline{\pmb{\theta}}$,
\end{itemize}
meaning that in both cases $\nu$ almost surely it can \textit{not} happen that starting with one individual of a given type the process dies out whereas starting from another type the process does not become extinct with positive probability.  

\begin{theorem}\label{z68}
    Let $L\geq 2$ and $\Sigma:=[L]^{\N}$ and $\nu $ is the uniform measure on $\Sigma$.

    Assume that for all $\theta\in [L]$ and $i\in [N]$ the elements ($\ell=0,\dots, N-1$) of the multivariate distributions $f_{\theta}^{(i)}$ are independent, and Binomial with parameters $k(\theta, i, \ell)\in\N$ and $p$. Namely, for $\mathbf{s}=(s_0,\dots, s_{N-1})\in[0,1]^{N}$, the all pgf has the following form: $f_{\theta}^{(i)}(\mathbf{s})=\prod_{\ell=0}^{N-1}(1-p+ps_{\ell})^{k(\theta, i, \ell)}$. 
    
    Assume further that $\mathcal{M}=\{\mathbf{M}_0, \dots, \mathbf{M}_{L-1}\}$ is a good set matrices in a sense of Definition \ref{x52}. 
    Then 
    \begin{enumerate}
		\item If $\lambda > 0$ then $\mathbf{q}(\overline{\pmb{\theta}})<\mathbf{1}$ for $\nu$-almost every $\overline{\pmb{\theta}}\in \Sigma$.
		\item If $\lambda < 0$ then $\mathbf{q}(\overline{\pmb{\theta}})=\mathbf{1}$ for $\nu$-almost every $\overline{\pmb{\theta}}\in \Sigma$.\label{x48}
		\item If $\lambda=0$ and there exists a $\theta\in[L]$ such that for every type $i$ the probability that a type $i$ individual gives birth to only $0$ or $1$ child is less than $1$ then $\mathbf{q}(\overline{\pmb{\theta}})=\mathbf{1}$ for $\nu$-almost every $\overline{\pmb{\theta}}\in \Sigma$.\label{x47}
	\end{enumerate}
\end{theorem}

\section{The MBPRE of a CISSIFS}\label{x81}
For a less formal explanation through an illustrative example see \cite[Section 1.3]{2024MBPRE}. We use the notation of Section \ref{z22}.
In what follows we formally define the MBPRE corresponding to a given CISSIFS (defined in Section \ref{a77}) corresponding to the integer self-similar IFS $\mathcal{S}$ with contraction ratio $L^{-1}$ and basic intervals $J^{(U)}$, $U\in[N]$ (see Section \ref{z45}). 
\begin{itemize}
    \item The \textit{environments} are denoted by $\overline{\pmb{\theta}}=(\theta _1,\theta_2,\dots  )\in\Sigma =[L]^{\mathbb{N}}$.
   We choose a random environment according to
   an ergodic invariant measure $\nu$ on $(\Sigma ,\sigma)$. Most commonly we choose
   the uniform measure $\nu :=\left( \frac{1}{L},\dots  ,\frac{1}{L} \right)^{\mathbb{N}}$ on $\Sigma$ to investigate the Lebesgue measure of the attractor of the CISSIFS.
   \item There are $N$ \textit{types}, the  type space is identified with $[N]$. Each type corresponds to the image of the interval $J^{(V)}$, $V\in[N]$ in a way described below.
\end{itemize}
\subsection{The multitype branching process description of a CISSIFS}\label{x42}
We fix an environment $\overline{\pmb{\theta}}=(\theta _1,\theta _2,\dots  )\in\Sigma $, and $U\in [N]$ (the index of the basic interval we start with).
On level  $0$ we have $1$ individual of type-$U$ for a $U\in[N]$. That is
$\mathbf{Z}_0=\mathbf{e}_{U}$.

 We denote types by indices, namely, we define the level $n\geq 1$ individuals of type-$V$ in environment $\pmb{\theta}$ in the basic interval $J^{(U)}$: 
\begin{equation} 
\label{a89}
\mathcal{X}_{U,\overline{\pmb{\theta}},n}^{V} :=
\left\{
    \mathbf{i} \in \mathcal{E}_n:
\
S_{\mathbf{i}}(J^{(V)})=J _{\overline{\pmb{ \theta }}|_n}^{(U) }
\right\}.
\end{equation}
Let $\mathbf{i}\in \mathcal{X}_{U,\overline{\pmb{\theta}},n}^{V}$.
Then the type $W\in[N]$ offspring of the type-$V$ individual $\mathbf{i}$ in the environment 
$\overline{\pmb{\theta}}$, in the basic interval $J^{(U)}$  is 
\begin{equation}
\label{a88}
\mathcal{Y}^{(V)}_{\mathbf{i},U,\overline{\pmb{\theta}}|_{n+1}}(W):=
\left\{  \mathbf{i}i_{n+1}\in
\mathcal{E}_{n+1}
:\
S_{\mathbf{i}i_{n+1}}(J^{(W)})=
J _{\pmb{ \theta }|_{n+1}}^{(U)}
\right\}.
\end{equation}
Moreover, for an $\mathbf{i}\in \mathcal{X}_{U,\overline{\pmb{\theta}},n}^{V}$ let
\begin{equation}
\label{a87}
Y^{(V)}_{\mathbf{i},U,\overline{\pmb{\theta}}|_{n+1}}(W):=
\#
\mathcal{Y}^{(V)}_{\mathbf{i},U,\overline{\pmb{\theta}}|_{n+1}}(W).
\end{equation}
Then by definition,
\begin{equation}
\label{a83}
\bigcup \limits_{V\in[N]}
\bigcup \limits_{\mathbf{i}\in \mathcal{X}_{U,\overline{\pmb{\theta}},n}^{V}} 
\mathcal{Y}^{(V)}_{\mathbf{i},U,\overline{\pmb{\theta}}|_{n+1}}(W)
= 
\mathcal{X}_{U,\overline{\pmb{\theta}},n+1}^W.
\end{equation}

Observe that
\begin{align}\label{a85} 
    \text{If }
&\mathbf{i}\in \mathcal{X}_{U,\overline{\pmb{\theta}},n}^{V},\text{then }\\
&\mathbf{i}i_{n+1}\in\mathcal{X} _{U,\overline{\pmb{ \theta }},n+1}^W \Longleftrightarrow 
\left(
 S_{i_{n+1}}(J^{(W)})=J _{\theta _{n+1}}^{(V) } \quad \text{ and }\quad  \mathbf{i}i_{n+1}\in \mathcal{E}_{n+1}\right).
\end{align}
Note that the  event in the bracket in the previous line is independent for different $\mathbf{i}\in \mathcal{X}_{U,\overline{\pmb{\theta}},n}^{V}$, and it has probability $p$  if $S_{i_{n+1}}(J^{(W)})=J _{\theta _{n+1}}^{(V) }$. That is for every $\mathbf{i}\in [M]^n$
\begin{multline}\label{a84} 
\mathbb{P}\left(
    \mathbf{i}i_{n+1}\in
\mathcal{X}_{U,\overline{\pmb{\theta}},n+1}^W|
\mathbf{i}\in \mathcal{X}_{U,\overline{\pmb{\theta}},n}^V
    \right)
\\
=
\mathbb{P}\left(
        i_{n+1}\in 
        \mathcal{X}^{W}_{V,\theta _{n+1},1}
     \right) 
     =
\left\{
\begin{array}{ll}
p
,&
\hbox{if $S_{i_{n+1}}(J^{(W)})=J _{\theta _{n+1}}^{(V)}$  ;}
\\
0
,&
\hbox{otherwise.}
\end{array}
\right.
\end{multline}
Recall that the event $\{k\in \mathcal{X}^{W}_{V,\theta _{n+1},1}\}$ coincides with the event
$ 
\{ 
    k\in\mathcal{E}_{1}, \, S_{k}(J^{(W)})=J _{\theta _{n+1}}^{(V) } \}$, for every $k\in [M]$. 
    Now we shall recall the definition of the $\mathbf{B}$ matrices from  \eqref{a90} to see that 
    \eqref{a84}
 implies that for $\mathbf{i}\in \mathcal{X}_{U,\overline{\pmb{\theta}},n}^{V}$ we have
\begin{equation}
\label{a82}
\mathbb{E}\left(
    Y^{(V)}_{\mathbf{i},\overline{\pmb{\theta}}|_{n+1}}(W)
\right)=p \mathbf{B}_{\theta _{n+1}}(V,W).
\end{equation} 
Now for every
$\mathbf{i}\in \mathcal{X}_{U,\pmb{\theta},n}^{V}$
we form the random vectors
\begin{equation}
\label{a81}
\mathbf{Y}^{(V)}_{\mathbf{i},\overline{\pmb{\theta}}|_{n+1}}:=
\left(
    Y^{(V)}_{\mathbf{i},\overline{\pmb{\theta}}|_{n+1}}(0),\dots  ,
    Y^{(V)}_{\mathbf{i},\overline{\pmb{\theta}}|_{n+1}}(N-1)
 \right).
\end{equation}
By definition, these random vectors are independent and are independent of $\mathcal{X}_{U,\pmb{\theta},n}^{V}$ given $\mathbf{i}\in \mathcal{X}_{U,\pmb{\theta},n}^{V}$.
We consider
\begin{equation} 
\label{a80}
\mathbf{Z}_n^{(U)}(\overline{\pmb{\theta}})=(Z_n^{(U)}(\overline{\pmb{\theta}})(0),\dots, Z_n^{(U)}(\overline{\pmb{\theta}})(N-1)):=
\sum_{V\in[N]}
\sum_{\mathbf{i}\in \mathcal{X}_{U,\pmb{\theta},n}^{V} }
\mathbf{Y}^{(V)}_{\mathbf{i},\overline{\pmb{\theta}}|_{n+1}}.
\end{equation}
Then $\mathcal{Z}^{(U)}(\overline{\pmb{\theta}})=\left\{ \mathbf{Z}_n^{(U)} (\overline{\pmb{\theta}})\right\}_{n=0}^{ \infty  }$ is a multitype branching process in a varying environment with initial distribution $\mathbf{Z}_0=\mathbf{e}_{U}$. For this, given $\nu$, we can define the corresponding MBPRE $\mathbf{Z}^{(U)}$ for each $U\in[N]$.
 This process corresponds to the $U$-th interval ($J^{(U)}$) of CISSIFS $\Lambda_{\mathcal{S}, p}$. The correspondence can be described as follows. Recall that for a $U\in [N]$ we defined the projection 
$\Gamma^{(U)}(\overline{\pmb{\theta}}):= b_UL+\sum_{\ell =1}^{\infty }
\theta _\ell L^{-(\ell -1)}$ ($b_UL$ is the left-endpoint of the interval $J^{(U)}$).
This projection connects the environments of the MBPRE $\mathcal{Z}^{(U)}$ to the points of the attractor $\Lambda_{\mathcal{S}, p}$ contained in $J^{(U)}$.
Assume that we are given an $x\in \bigcup_{U\in [N]} J^{(U)}$. Then there exist a coding given by the index of the basic interval containing $x$
(say $U\in[N]$) and the $L$-adic coding of $x$ "inside" the interval ($\overline{\pmb{\theta}}$), namely there exists $U \in [N]$ and a $ \overline{\pmb{\theta}}=(\theta_1, \theta_2, \dots)\in \Sigma$ such that $x=\Gamma^{(U)}(\overline{\pmb{\theta}})$,
which coding is $\nu$-mod 0 unique. 
Clearly, by the definition of the process $\mathcal{Z}^{(U)}$, if the process does not die out in the varying environment $\overline{\pmb{\theta}}$ (for a given realization $\pmb{\omega}$) then $x\in \Lambda_{\mathbf{S}, p}(\pmb{\omega})$.

It follows from \eqref{a82}
that the expectation matrices for $\mathcal{Z}$ are
\begin{equation}
    \label{a79}
    \mathbf{M}_\theta :=p\cdot \mathbf{B}_\theta ,\quad \text{ for }
    \theta \in[L].
    \end{equation}
By Theorem \ref{z68} if
$\mathcal{M}=\left\{ \mathbf{M}_0,\dots  ,\mathbf{M}_{L-1} \right\}$
is good w.r.t. $\nu $,  and the Lyapunov exponent ($\lambda_{\mathcal{M}, \nu}$) corresponding to the expectation matrices is positive  then for $\nu$-almost every $\overline{\pmb{\theta}}\in \Sigma$ the processes $\mathcal{Z}^{(U)}$ does not die out with positive probability (i.e. $\mathbf{q}({\overline{\pmb{\theta}}})\neq \mathbf{1}$).
 Hence, there is an index $U\in [N]$
 for which we can find a
  $\widehat{\Sigma}\subset \Sigma$, with $\nu(\widehat{\Sigma})>0$ such that for all $\overline{\pmb{\theta}} \in \widehat{\Sigma}$, we have $q^{(U)}(\overline{\pmb{\theta}})\neq 1$.
 This means that $x=b_UL+\sum_{\ell =1}^{\infty }
\theta _\ell L^{-(\ell -1)}$, as explained above, will be contained in $\Lambda_{\mathcal{S}, p}$ with positive probability. In this way we have proved the following lemma:
\begin{lemma}\label{a78}
    Assume that we have a CISSIFS with contraction ratio $L^{-1}$ as it was defined in Section \ref{a77}, an ergodic measure $\nu $ on $(\Sigma ,\sigma )$, where $\Sigma :=[L]^{\mathbb{N}}$. The
     basic intervals are $J^{(k)}$, $k\in [N]$ (see $\eqref{a92}$) and the expectation matrices are
     $\mathbf{M}_{\theta}=p\mathbf{B}_{\theta}$, $\theta \in[L]$ (see \eqref{a90}). Assume that
     $\mathcal{M}=\{\mathbf{M}_0, \dots, \mathbf{M}_{L-1}\}$ is good w.r.t. $\nu $.
    Then there exist a set $K\subset \bigcup_{U\in[N]}J^{(U)}$
    of positive $\widetilde{\nu}$-measure (\eqref{a69})
     such that
    \begin{equation}
        \mathbb{P}(x\in \Lambda_{\mathcal{S}, p})>0, \quad
        \text{ for every } x\in K.
    \end{equation}
\end{lemma}
In a special case of the uniform measure it was shown in \cite{OurPaper} that this implies that $\widetilde{\nu}(\Lambda_{\mathcal{S}, p})>0$ almost surely conditioned on non-extinction. For the convenience of the reader we present the proof in this more general case.

\begin{lemma}\label{d77}
    \begin{enumerate}
        \item
    If there exist a set $K$   of $\widetilde{\nu}$-positive measure  such that for $x\in K$:
        $\mathbb{P}(x\in\Lambda_{\mathcal{S},p})>0$, then
        \begin{equation}\label{d98}
            \mathbb{P}(\widetilde{\nu}(\Lambda_{\mathcal{S},p})>0)>0.
        \end{equation}
        \item If for $\widetilde{\nu}$ almost every $x$:
        $\mathbb{P}(x\in\Lambda_{\mathcal{S},p})=0$, then
        \begin{equation}\label{q95}
            \mathbb{P}(\widetilde{\nu}(\Lambda_{\mathcal{S},p})>0)=0.
        \end{equation}
    \end{enumerate}
    \end{lemma}
    \begin{proof}
        \eqref{d98} holds if and only if $\mathbb{E}(\widetilde{\nu}(\Lambda_{\mathcal{S},p}))>0$, and \eqref{q95} if and only if $\mathbb{E}(\widetilde{\nu}(\Lambda_{\mathcal{S},p}))=0$. Observe that
    \begin{align*}
        \mathbb{E}\left(\widetilde{\nu}\left(\Lambda_{\mathcal{S}, p}\right)\right)
         &=\int\limits_{\Omega}\widetilde{\nu }\left(\Lambda_{\mathcal{S},p}\left( \pmb{\omega}\right)\right)d\mathbb{P}\left(\pmb{\omega}\right)  =
        \int\limits_{\Omega} \int\limits_{I} \ind{\left\{x\in \Lambda_{\mathcal{S},p}\left( \pmb{\omega}\right)\right\}}d\widetilde{\nu}\left(x\right)d\mathbb{P}\left(\pmb{\omega}\right) \\
        & =
        \int\limits_{I}\int\limits_{\Omega}\ind{\left\{x\in \Lambda_{\mathcal{S},p}\left( \pmb{\omega}\right)\right\}} d\mathbb{P}\left(\pmb{\omega}\right)d\widetilde{\nu}\left(x\right)  =
        \int\limits_{I}\mathbb{P}\left(x\in \Lambda_{\mathcal{S}, p}\right)d\widetilde{\nu}\left(x\right).
    \end{align*}
    The assertion of the first part of the lemma, $\widetilde{\nu}(\{x:\mathbb{P}(x\in\Lambda_{\mathcal{S},p})>0\})>0$, implies that $\mathbb{E}(\widetilde{\nu}(\Lambda_{\mathcal{S},p}))>0$. While the assertion of the second part $\widetilde{\nu}(\{x:\mathbb{P}(x\in\Lambda_{\mathcal{S},p})>0\})=0$ implies that $\mathbb{E}(\widetilde{\nu}(\Lambda_{\mathcal{S},p}))=0$.
    \end{proof}

    \begin{proof}[Proof of Theorem \ref{z42}]
The first part of the theorem follows from the combination of Lemma \ref{a78}
and a standard $0-1$-type lemma on percolation sets (see for example \cite[Lemma 3.9]{OurPaper}).
\begin{comment}
noteing that it is immediate that the first two conditions of
\cite[Lemma 3.9]{OurPaper} hold. The third condition of \cite[Lemma 3.9]{OurPaper} is the assertion of Lemma \ref{d77}
\end{comment}

The second part is again the combination of Lemma \ref{d77} and the secodnd part of Theorem \ref{z68}.
    \end{proof}

\section{Non-existence of interior points and positive Lebesgue measure}\label{x16}
\subsection{Non-existence of interior points}
In this section we consider certain conditions under which a CISSIFS (defined in Section \ref{a77}) does not contain an interval. We follow the notation of \cite{jungers_2009}.
\begin{definition}[Lower spectral radius] \label{q91}
     Consider the set of matrices $\mathcal{B}=\{\mathbf{B}_0, \dots, \mathbf{B}_{L-1}\}$ and a sub-multiplicative matrix norm $\|\cdot\|_*$.
    \begin{equation*}
        \widecheck{\rho}_{n}(\mathcal{B}, \|\cdot\|_*):=\inf\{ \|\mathbf{B}_{i_1}\cdots \mathbf{B}_{i_n}\|_*^{\frac{1}{n}},\, \mathbf{B}_{i_j}\in \mathcal{B}, \: 1\leq j \leq n\}
    \end{equation*}
    then the \textit{lower spectral radius corresponding to $\mathcal{B}$} is
    \begin{equation*}
        \widecheck{\rho}(\mathcal{B})= \lim_{n\to \infty} \widecheck{\rho}_{n}(\mathcal{B}, \|\cdot\|_*),
    \end{equation*}
where the limit above exists, see \cite{jungers_2009}.
\end{definition}
Recall the notation of Section \ref{z43}, let $\mathcal{S}=\left\{ S_i \right\}_{i=0}^{M-1 }$ be an ISSIFS of the form of \eqref{a99}. Let $p\in(0,1]$. We write $\Lambda _{\mathcal{S},p}$ for the  corresponding coin tossing integer self-similar set defined in Section \ref{a77}.
Let $\mathcal{B}:=\left\{ \mathbf{B}_0,\dots  ,\mathbf{B}_{L-1} \right\}$, where the matrices $\mathbf{B}_\theta $, $\theta \in[L]$  were defined in \eqref{a90}.

\begin{proposition}\label{x94}
    If $p<(\widecheck{\rho}(\mathcal{B}))^{-1}$, then $\Lambda _{\mathcal{S},p}$ does not contain an interval almost surely.
\end{proposition}
\begin{proof}
    We use a standard argument similar to the one used in \cite{zbMATH05255663}. In the definition of the lower spectral radius (Definition \ref{q91}) for a non-negative matrix $\mathbf{A}$ let $\|\mathbf{A} \|_*=\|\mathbf{A}\|:=\sum_{i, j}A_{i,j}$.
    Since $p\cdot \widecheck{\rho}(\mathcal{B})<1$, we can choose an $\varepsilon >0$ such that $p\cdot \widecheck{\rho}(\mathcal{B})<1-2\varepsilon$. For this $\varepsilon$ there exists an $\widetilde{N}$, such that if $n>\widetilde{N}$, then
    \begin{equation}
        p\cdot \widecheck{\rho}_{n}(\mathcal{B}, \|\cdot\|)<(1-\varepsilon).
    \end{equation}
    Fix $m=\widetilde{N}+1$, for this $m$, by the previous observation, there exists a $\pmb{\theta}=(\theta_1, \dots, \theta_m)$ such that $p\|\mathbf{B}_{\theta_1}\cdots \mathbf{B}_{\theta_n}\|^{\frac{1}{n}}<(1-\varepsilon)$, hence for the expectation matrices,
    \begin{equation}
        \|\mathbf{M}_{\theta_1}\cdots \mathbf{M}_{\theta_n}\|<(1-\varepsilon)^n.
    \end{equation}
    Let $\pmb{\theta}^n$ denote the vector which we get by concatenating $\pmb{\theta}$ with itself $n$-times. Then clearly,
$\lim\limits_{n \to \infty}\left\|\mathbf{M}_{\pmb{\theta}^n}\right\|=0$.
By the submultiplicativity  of the matrix norm it follows that for any $\mathbf{c}_k=(c_1, \dots, c_k)\in [L]^k$:
\begin{equation}\label{y7}
    \lim_{n \to \infty} \left\|\mathbf{M}_{\mathbf{c}_k\pmb{\theta}^n}\right\|=0.
\end{equation}
Let $\mathbf{c}_{k}\in [L]^k$ be given, denote $\overline{\pmb{\theta}}=(\mathbf{c}_{k}, \pmb{\theta}, \pmb{\theta}, \dots)$ and recall that $\sum_{U\in[N]}\mathbf{Z}^{(U)}_{k+n}(\overline{\pmb{\theta}})$ denote the number of level $k+n$ cylinders intersecting
$\bigcup_{U \in [N]}J^{(U)}_{\mathbf{c}_k\pmb{\theta}^n}$.
\begin{equation*}
    \mathbb{E}\big(\sum_{U\in[N]}\mathbf{Z}^{(U)}_{k+n}(\overline{\pmb{\theta}})\big)= \|\mathbf{M}_{\mathbf{c}_k\pmb{\theta}^n}\|\to 0 \text{ as } n\to \infty ,
\end{equation*}
thus by Markov's inequality:
\begin{equation*}
    \mathbb{P}\big(\sum_{U\in[N]}\mathbf{Z}^{(U)}_{k+n}(\overline{\pmb{\theta}})\geq 1\big)\leq \mathbb{E}\big(\sum_{U\in[N]}\mathbf{Z}^{(U)}_{k+n}(\overline{\pmb{\theta}})\big)\to 0 \text{ as } n\to \infty.
\end{equation*}
In this way the points
\begin{equation*}
    \bigcup_{U\in [N]}\bigcap\limits_{n\geq 1}J^{(U)}_{\mathbf{c}_k\pmb{\theta}^n}
\end{equation*}
are not contained in $\Lambda_{\mathcal{S}, p}$ with probability one.
By varying $\mathbf{c}_k$ we get a countable dense subset of $\bigcup_{U\in [N]}J^{(U)}$ which is not contained in $\Lambda_{\mathcal{S}, p}$ with probability one, hence it can not contain an interval.
\end{proof}
\begin{comment}
\rk{\begin{remark}
    Note, that in the case when the Lower spectral radius is actually the spectral radius of one of the matrices (i.e. there exists an $i$ such that $\widecheck{\rho}(\mathcal{B})=\rho(\mathbf{B}_i)$, where $\rho(\mathbf{B}_i)$ is the spectral radius of the matrix $\mathbf{B}_i$ ) of the CISSIFS and further $p\neq 1$ then it follows from \cite[Theorem 10.1]{harris_branching}, that $p\leq\rho(\mathbf{B}_i)$ is sufficient in Proposition \ref{x94}.
\end{remark}}
\end{comment}
\subsection{Non-existence of interior points and positive Lebesgue measure}

In this section we examine the following question; \textit{under what conditions can we guarantee the existence of a parameter interval (of $p$) such that for any probability from the interval we have that the random attractor, $\Lambda_{\mathcal{S}, p}$ has positive Lebesgue measure almost surely, conditioned on non-extinction, but it does not contain interior points almost surely.} It follows from Proposition \ref{x94} and Theorem \ref{z42} that under the conditions of Theorem \ref{z42} if the Lyapunov exponent (see Definition \ref{x84}) with respect to the uniform measure (denoted by $\lambda$) is strictly greater than the logarithm of the lower spectral radius, see \ref{q91} (denoted by $\log \widecheck{\rho}$), then $(\e{-\lambda}, \widecheck{\rho}^{-1})$ is a proper parameter interval, for which the desired properties hold.
One example for such behaviour is Example \ref{u99}. In this case the lower spectral radius equals to $1$, hence $\log \widecheck{\rho}=0$ but recall that $0.3961<\lambda<0.3962$.

An example, where such parameter interval does not exist appears in \cite{Simon_Rams_2014}. From this work it follows that for any projections of homogeneous Mandelbrot percolations the parameter interval does not exists. 

A possible way to approach the problem of having empty interior and positive Lebesgue measure is to study the set of matrices $\mathfrak{B}=\{\mathbf{B}_0, \dots, \mathbf{B}_{L-1}\}$ determined by the self-similar integer IFS. The matrices we get when considering this special family of IFSs are always non-negative and their sum is a primitive matrix. The pressure functions of matrices with the above property are investigated in \cite{barany_rams_2014} and a more general class for example in \cite{Feng_Lau_2002}, \cite{Feng_2004}, \cite{Feng_2007}. The above-mentioned pressure function is the following:
\begin{equation}
    P(q):=\lim_{n\to \infty}\frac{1}{n}\log \sum_{\pmb{\theta}\in[L]^n}\|\mathbf{B}_{\pmb{\theta}}\|_{*}^q.
\end{equation}
Under the non-negativity and primitivity of the sum assumptions mentioned above the defining limit exists for all $q\in\mathbb{R}$, and the pressure function is continuous, convex and monotone increasing for all $q\in \R$ and continuously differentiable for $q>0$ (see for example \cite{barany_rams_2014}).
The relevance of the pressure for us described as follows.
Firstly the asymptote of the pressure function at $-\infty$ is the logarithm of the lower spectral radius (see \cite[Lemma 2.3 (b)]{Feng_2007}), namely:
\begin{equation}
    \log \widecheck{\rho}(\mathfrak{B})=\lim_{q\to -\infty}\frac{P(q)}{q}.
\end{equation}
On the other hand it has been shown in \cite[Proof of Lemma 4.9.]{barany_rams_2014}, that the right derivative of the pressure function at $0$ is the Lyapunov exponent.
\begin{lemma}\label{x83}
    Assume that the set of matrices $\mathcal{B}$ is good (see Definition \ref{x52}) with respect to the uniform measure on $\Sigma$. If the negative part of the pressure function is not a straight line, then there exists a parameter interval $(p_0, p_1)$ such that for every $p_0<p<p_1$ the random attractor $\Lambda_{\mathcal{S}, p}$ has positive Lebesgue measure almost surely conditioned on non-extinction, but it does not contain an interior point almost surely.
\end{lemma}
\subsubsection{Partial solution}

In some special cases it is known from Theorem \ref{x90} below that the negative part of the pressure function is strictly convex. The Theorem requires that all matrices are invertible, which in general does not hold for the matrices defined in \eqref{a90} in section \ref{z45} corresponding to integer self-similar IFSs on the line. See for example Example 2.2 in \cite{OurPaper} corresponding to the projection of the Menger sponge to the space diagonal of the unit cube in $\mathbb{R}^3$ (note however that in this case we know that the interesting parameter interval exists by different considerations). In some cases however the conditions of the theorem are satisfied (see later Example \ref{x88}). 
Before stating the theorem and present the examples, we first introduce some notions, mostly based on Viana's lecture on Lyapunov exponents \cite{viana_lectures}.

Recall that we denote by $(\Sigma, \sigma)$ the (one-sided) shift space  over the finite alphabet $[L]$.

Let $ \mathbb{M}(\mathbb{R}, N)$ denote the set of $N\times N$ real matrices. For an $A:\Sigma\to \mathbb{M}(\mathbb{R}, N)$ we define the matrix cocycle 
\begin{equation}
    A(\overline{\pmb{\theta}}, n):=A(\sigma^{n-1}(\overline{\pmb{\theta}}))\cdots A(\overline{\pmb{\theta}}),
\end{equation}
for $\overline{\pmb{\theta}}\in\Sigma$ and $n\in \mathbb{N}$.
From now on we restrict ourselves from $A:\mathbb{M}(\mathbb{R}, N) \to \mathbb{M}(\mathbb{R}, N)$ to $A:\GL(N, \mathbb{R})\to \GL(N, \mathbb{R})$, cocycles of the $N\times N$ invertible, real matrices.

We say that $A$ is a \textit{one-step cocycle}, if $A$ only depends on the first letter of $\overline{\pmb{\theta}}$. In this case there is a natural correspondence between one-step cocycles and the elements of the set $\mathscr{G}_{N, L}=\{\mathfrak{B}=\{\mathbf{B}_0, \dots, \mathbf{B}_{L-1}\},\mathbf{B}_{i}\in  \mathbb\GL(N, \mathbb{R})\}$, hence for $\overline{\pmb{\theta}}=(\theta_1, \theta_2, \dots)$
\begin{equation}
    A(\overline{\pmb{\theta}},n)=\mathbf{M}_{\theta_n} \cdots \mathbf{M}_{\theta_1}.
\end{equation}
In this case we write $A\sim \mathfrak{B}$.

Given a matrix $\mathbf{B}\in\GL(N, \mathbb{R})$,
\begin{notation}
    \begin{enumerate}
    \item Let $\sigma_1\geq \dots \geq \sigma_N$ the singular-values of the matrix $B$. The eccentricity of the matrix is,
    \begin{equation}
    \Ecc(\mathbf{B})=\min_{1\leq\ell<N} \frac{\sigma_\ell}{\sigma_{\ell+1}}.
    \end{equation}
    \item The Grassmanian manifold is denoted by $\Grass(\ell, N)$.
    \end{enumerate}
\end{notation}

\begin{definition}[Pinching, Twisting]\label{x18}
    The one-step cocycle $A\sim\mathfrak{B}=\{\mathbf{B}_1, \dots, \mathbf{B}_L\}$ is
    \begin{enumerate}
        \item
         \textit{pinching} if there exists a product $\mathbf{B}_{i_1}\cdots \mathbf{B}_{i_n}$, $\mathbf{B}_{i_j}\in \mathfrak{B}$ with arbitrarily large eccentricity $\Ecc(\mathbf{B}_{i_1}\cdots \mathbf{B}_{i_n})$.
         \item
         \textit{Twisting} if for any $F \in \Grass(\ell, d)$ and any finite family $G_1, ..., G_K$ of elements of $\Grass(d-\ell, d)$ there exists a product $\widehat{B}=\mathbf{B}_{i_1}\cdots \mathbf{B}_{i_n}$, $\mathbf{B}_{i_j}\in \mathfrak{B}$ such that $\widehat{B}(F) \cap G_i = \{0\}$ for all $i=1, \dots, K$.
    \end{enumerate}
\end{definition}

\begin{definition}[1-typical]
    A one-step cocycle is 1-typical if it is \textit{pinching} and \textit{twisting}.
\end{definition}

\begin{corollary}[Corollary of {\cite[Theorem 9.1]{Tom_rush_2023}}]
    \label{x90}
    Consider a 1-typical one-step cocycle, then the corresponding pressure function is either affine in its domain or strictly convex in a neighbourhood of 0.
\end{corollary}

It was proven in \cite{barany_rams_2014} that the pressure function of the matrices corresponding to rational projections of 2-dimensional sponges is not a straight line in its whole domain when the reciprocal of the contraction ratio is not a divisor of the number of retained level 1 cylinders. Their proof could be generalized for the case of general matrices corresponding to CISSIFSs, this proof can be found in the Appendix of \cite{OurPaper}.

\begin{example}[$45^\circ$ projection of the Sierpiński carpet]\label{x88}
    Recall from Section \ref{a84} the expectation matrices corresponding to the $45^\circ$ projection of the Sierpiński carpet are the following.
It follows from \cite[Exercise 1.7]{viana_lectures} that the monoid corresponding to the matrices is pinching and twisting. Namely, $\mathbf{B}_0$, $\mathbf{B}_1$ and $\mathbf{B}_2$ are invertible, $2\times 2$ matrices with the following property: $\mathbf{B}_0$ and $\mathbf{B}_2$ are hyperbolic (i.e. they both have two different eigenvalues), hence the monoid is pinching, and they also don't share an eigensubspace, thus it is twisting as well. Hence the assumptions of Theorem \ref{x90} are satisfied. Consequently, in this case the interesting parameter interval exists.

\end{example}

\begin{comment}
\begin{example}[Example \ref{x57} continued]\label{x58}
    The two matrices from the example are invertible.
    By \cite[Exercise 8.4]{viana_lectures} similarly to the 2-dimensional case in higher dimension the monoid corresponding to $\mathbf{B}_0$ and $\mathbf{B}_1$ is \begin{itemize}
        \item pinching, if there exists a $\mathbf{A}_1$ from the monoid such that all of its eigenvalues are different in norm, and
        \item twisting, if for $\mathbf{A}_1$ above we can find an $\mathbf{A}_2$ from the monoid such that
        $$
        \mathbf{A}_2(V)\cap W=\{0\},
        $$
        for any pair $V,W$ of $\mathbf{A}_1$ invariant subspaces with complementary dimension.
    \end{itemize}
    It can be checked (for example using Wolfram Mathematica) that the above hold for the choices $\mathbf{A}_1=\mathbf{B}_0^{2}$ and $\mathbf{A}_2=\mathbf{B}_1^{2}$.
\end{example}
\end{comment}
\begin{example}[Example \ref{x20} continued]\label{x24}
    The two matrices are invertible.
    It can be  checked that the conditions mentioned in the previous example hold for the choice $\mathbf{A}_1=\mathbf{B}_0\cdot \mathbf{B}_1$ and $\mathbf{A}_2=\mathbf{B}_1^{2}$
\end{example}

\section{Existence of interior points}
In \cite{OurPaper} we gave a condition on the expectation matrices under which the random attractor contained an interior point almost surely conditioned on non-extinction. Namely, we required that all the expectation matrices has all column sums greater than 1. In some cases this turned out to be the sufficient condition (see for example \cite[Theorem 2.1]{OurPaper} about the Menger-sponge) but in general this is not sufficient as we show in case of Example \ref{x20}.

\begin{comment}
\begin{example}[Example \ref{x57} continued]\label{x56}
    In the case of Example \ref{x57} $\mathbf{M}_0$ and for all $n\in\N$, $\mathbf{M}_0^{n}$ has first column $(0, 0, p)^{n}$, hence the condition only holds with $p=1$. 
Note that the spectral radius of $\mathbf{M}_0$ is $\frac{1}{2}(1+\sqrt{5})$, allowing the possibility of existence of a $p<1$ which choice leads to existence of interior point (a.s. on non-extinction).
\end{example}
\end{comment}

We will show that this is indeed the case here using the following proposition.

Let $\mathfrak{V}(N)= \{\mathbf{u}\in \mathbb{N}^N\}$, the set of non-negative integer vectors of length $N$.
\begin{proposition}\label{x98}
    Consider the CISSIFS with attractor $\Lambda_{\mathcal{S}, p}$ and expectation matrices $\mathcal{M}=\{\mathbf{M}_0=p\cdot \mathbf{B}_0, \dots, \mathbf{M}_{L-1}=p\cdot\mathbf{B}_{L-1}\}$. 
    Assume that there exists a non-empty set $\mathcal{U}:=\{\mathbf{u}_1, \dots, \mathbf{u}_m\}\subset\mathfrak{V}(N)\setminus\{\mathbf{0}\}$
    \begin{enumerate}
    \item there exists $\mathbf{u}^*\in \mathcal{U}$ and there exists $\pmb{\theta}^{*}\in [L]^{S^{*}}$ for some $S^{*}\geq 1$ and an $U^{*}\in [N]$ such that 
    \begin{enumerate}[label=(\alph*)]
        \item  \begin{equation}\label{x66}
            \mathbf{e}_{U^{*}}^T\mathbf{B}_{\pmb{\theta}^*}\geq \mathbf{u}^*.
        \end{equation}
        \item For the deterministic attractor $\Lambda_{\mathcal{S}}$, $\Lambda_{\mathcal{S}}\cap J^{(\pmb{\theta}^{*})}$ contains an interior point. (Note that this holds when the deterministic attractor is itself an interval.)
    \end{enumerate}
   
    \item
    Further, there exists a $\gamma'>1$ and a level $S'$ such that for all $\pmb{\theta}\in [L]^{S'}$ there exists a non-negative, $|\mathcal{U}|\times |\mathcal{U}|$ matrix $\mathbf{A}_{\pmb{\theta}}$ with all row sums greater than $\gamma'$ (i.e. for all $i\in[|\mathcal{U}|]$ $\sum_{k\in[|\mathcal{U}|]}\mathbf{A}_{\pmb{\theta}}(i,k)>\gamma'>1$). Assume that for this $\mathbf{A}_{\pmb{\theta}}$
    \begin{equation}
        \mathbf{U}\mathbf{M}_{\pmb{\theta}}\geq \mathbf{A}_{\pmb{\theta}}\mathbf{U},
    \end{equation}
    where $\mathbf{U}$ is the $|\mathcal{U}|\times N$ matrix having row vectors $\mathbf{u}_i^T$ for $i=1, \dots, m$.

    Then $\Lambda_p$ contains an interval almost surely conditioned on non extinction.
    \end{enumerate}
\end{proposition}
In what follows we give Condition $2^*$, an alternative to Condition $2$ which are shown to be equivalent later in Lemma \ref{x31}. Then we state a Corollary of the Proposition, about some cases when the existence of interior point in the projection depends on the lower spectral radius corresponding to the expectation matrices. The corollary states that this surely happens, when we can get the lower spectral radius by instead of using a submultiplicative matrix norm in the definition we use the supermultiplicative minimum columns sum. Then in Application \ref{x22} we use our Proposition on Example \ref{x20}, and finally (before the proofs) in Remark \ref{x12} we explain the intuitive meanings of both Conditions $2$ and $2^*$.  

\begin{definition}[Condition $2^*$]
    There exists a $\gamma>1$ and a level $S$ such that for all $\mathbf{u}\in \mathcal{U}$ and for all $\pmb{\theta}\in [L]^S$ there exists a $\mathbf{v}\in\mathcal{U}$ such that 
    \begin{equation}
        \mathbf{u}^T\mathbf{M}_{\pmb{\theta}}\geq \gamma \mathbf{v}.
    \end{equation}
\end{definition}
\begin{corollary}[of Proposition \ref{x98}]\label{x15}
    If for the expectation matrices $\mathcal{M}$ of the CISSIFS it holds that 
    \begin{equation}\label{x14}
        1<\widecheck{\rho}(\mathcal{M})= \lim_{n\to \infty} \widecheck{\rho}_{n}(\mathcal{M}, (\cdot)_*), 
    \end{equation}
    where $(\mathbf{M})_*$ is the minimal column sum ($(\mathbf{M})_*=\min_{j}\sum_{i}\mathbf{M}(i,j)$) 
    of the matrix $\mathbf{M}$ then Condition $2$ of Proposition \ref{x98} is satisfied for $\mathcal{U}=\{(1, \dots, 1)\}$.
\end{corollary}
\begin{proof}[Proof of Corollary \ref{x15}]
    If \eqref{x14} holds, it means, that there exists an $n$ such that for all $\pmb{\theta}\in[L]^{n}$ we have 
    \begin{equation}
        \mathbf{e}^T\mathbf{M}_{\pmb{\theta}}\geq \alpha \mathbf{e}^{T},
    \end{equation}
    for some $\alpha>1$. 
\end{proof}

\begin{lemma}\label{x31}
    Condition 2 and Condition $2^*$ are equivalent (the parameters $\gamma$ and $S$ might differ).
\end{lemma}

\begin{application}[Example \ref{x20}]\label{x22}
    In both examples we estimated the probability using Wolfram Mathematica. 
    In case of Example \ref{x20} 
    $$
    \mathfrak{U}=\{(1, 0, 1,0), (0, 1, 0, 1)\},
    $$
in this case our estimation for the critical probability is $\widehat{p}\leq 377^{-1/13}\sim 0.633607$.
\end{application}

\begin{remark}[The meaning of the assumptions in Proposition \ref{x98}]\label{x12}
    The set of vectors $\mathcal{U}$ contains some possible arrangement of types, e.g. $(1, 0, 1)$ would mean we have one of the type 0 and one of the type 2 individuals and zero of the type-1 individuals. The first condition guarantees that with positive probability we can reach one of the possible arrangements from $\mathcal{U}$. 
    Before explaining the meaning of Condition 2 and $2^*$ we first explain the key idea behind these types of proofs which appears in multiple papers (e.g. \cite{zbMATH00064453}, \cite{zbMATH05255663}, \cite{zbMATH05873843},\cite{Simon_Rams_2014} and \cite{OurPaper}). First think of the proof of Falconer and Grimmett (\cite{zbMATH00064453}). In this paper they consider the projection to the coordinate axis, hence they have exactly one type. In this case there exists interval in the random attractor (with probability one conditioned on non-extinction) iff in all environments the expected number of individuals are greater than 1. In \cite{zbMATH05255663} and \cite{OurPaper} we already have multiple types, in these cases what we ask for is that in every environment if we start the process with one individual from each type, then from each type the expected number of individuals will be greater than one. This is phrased as: all column sums of all expectation matrices are greater than 1. This requirement is a special case of condition $2^*$ (hence also of Condition 2), namely in the above case condition $2^*$ is satisfied with $\mathcal{U}=\{(1, \dots, 1)\}$. Intuitively this means that if we glue together all the types and consider that as a "meta"-type, then the expected growth of the number of individuals of this "meta"-type is greater than one.
    This is relaxed in a way that we ask for similar growth in every environment (in expectation) however we ask for a set of vectors such that for each vector from the set the growth is attained. Which can be phrased as, we have multiple "meta"-types (arrangement of actual types), and we require that in each environment each "meta"-type expected behaviour is that it gives birth to more than one individual of one given "meta"-type (Condition $2^*$).
    
    Now, condition 2 is a very natural generalization of this in a way, that in expectation it gives birth to more than one individual of all the "meta"-types together.
\end{remark}
In what follows we present the proofs of the statements.
\begin{proof}[Proof of Lemma \ref{x31}]
    It is easy to see that Condition $2^*$ implies Condition 2, in this case the matrix $\mathbb{A}_{\theta}$ has exactly one positive element in each row.
    Condition $2^*$ can be rephrased in the following way:
    Assume that there exists a level $S$ such that for all $\pmb{\theta}\in [L]^{S'}$ there exists a non-negative $|\mathcal{U}|\times |\mathcal{U}|$ matrix $\mathbf{A}_{\pmb{\theta}}$ with the property that every row of 
    $\mathbf{A}_{\pmb{\theta}}$ contains exactly one element which element is greater than $\gamma$.
    To prove that this 2 implies $2^*$ choose $S$ so that $\gamma'^S/|\mathcal{U}|>\gamma$. Then since the smallest column sum of a product of matrices is greater than the product of the smallest column sums (see Fact \ref{x21}) we get that for any $\pmb{\theta}=\pmb{\theta}_1\dots \pmb{\theta}_{S}\in [L]^{S'\cdot S}$:
    \begin{equation}
        \mathbf{U}\mathbf{M}_{\pmb{\theta}}\geq \mathbf{A}_{\pmb{\theta}_1}\mathbf{U}\mathbf{M}_{\pmb{\theta}_2}\dots\mathbf{M}_{\pmb{\theta}_S}\geq  \mathbf{A}_{\pmb{\theta}_1}\dots\mathbf{A}_{\pmb{\theta}_S}\mathbf{U}.
    \end{equation}
    Since all row sums of $\mathbf{A}_{\pmb{\theta}_i}$ are greater than $\gamma'$ it follows that the product has all row sums greater than $\gamma'^{S}$ from which it follows that each row has at least one element which is greater than $\gamma'^S/|\mathcal{U}|>\gamma$,  by the choice of $S'$, which proves that the required condition follows.
\end{proof}
Now we state (without proof) a very simple fact. 
\begin{fact}\label{x21}
    \begin{enumerate}
        \item $\mathbf{A}_1\geq\mathbf{A}_2$ implies that for any non-negative matrix $\mathbf{B}$, $\mathbf{A}_1\cdot \mathbf{B}\geq\mathbf{A}_2\cdot \mathbf{B}$.
        \item The minimal row-sum is super-multiplicative for non-negative matrices.
    \end{enumerate}
\end{fact}

\begin{fact}[Large deviation theorem]\label{x78}
    Assume that the random variables $\mathbf{S}_1, \dots, \mathbf{S}_{\ell}$ are independent and identically distributed and has the same distribution as $\mathbf{S}$, with $\mathbb{E}(\mathbf{S})=\gamma>\eta$. There exists a $0<\delta<1$ such that 
    $$
    \mathbb{P}(\mathbf{S}_1+ \dots+ \mathbf{S}_{\ell}\leq \eta\cdot \ell)\leq \delta^{\ell}
    $$
\end{fact}

In the proof of the Proposition \ref{x98} we use the notation of Section \ref{x42}. By a very similar argument as in the beginning of Section \ref{x42} analogously to \eqref{a80} it can be shown that 
\begin{lemma}\label{x41}
For all fix $U\in[N]$ and $\overline{\pmb{\theta}}\in\Sigma$ we have for all $W\in[N]$
\begin{equation}
    Z_{n+k}^{(U)}(\overline{\pmb{\theta}})(W)=\sum_{V\in[N]}\sum_{\mathbf{i}\in \mathcal{X}_{U,\pmb{\theta},n}^{V} }
    Y^{(V)}_{\mathbf{i},\overline{\pmb{\theta}}|_{n+k}, k}(W),
\end{equation}
where
\begin{itemize}
    \item $Y^{(V)}_{\mathbf{i},\overline{\pmb{\theta}}|_{n+k}, k}(W)$ are jointly independent random variables for $V\in [N],\mathbf{i}\in \mathcal{X}_{U,\pmb{\theta},n}^{V}, W\in[N]$,
    \item they are also independent of $\mathcal{X}_{U,\pmb{\theta},n}$ and
    \item for $V\in [N],\mathbf{i}\in \mathcal{X}_{U,\pmb{\theta},n}^{V}, W\in[N]$ 
     $$
     Y^{(V)}_{\mathbf{i},\overline{\pmb{\theta}}|_{n+k}, k}(W)\stackrel{d}{=}Z^{(V)}_k(\sigma^{n}(\overline{\pmb{\theta}}))(W).
     $$
\end{itemize}
Conditioned on $\mathcal{X}^{V}_{U, \overline{\pmb{\theta}}, n}$ we usually write
$$
Z_{n+k}^{(U)}(\overline{\pmb{\theta}})(W)=\sum_{V\in[N]}\sum_{j=1}^{Z_{n}^{(U)}(\overline{\pmb{\theta}})(V)}
    Y^{(V)}_{j,\overline{\pmb{\theta}}|_{n+k}, k}(W).
$$
\end{lemma}
\begin{proof}
    We will prove that $\Lambda_p$ contains an interval almost surely conditioned on non extinction, by proving that it contains an interval with positive probability and similar to the proof of Theorem \ref{z42} we use \cite[Lemma 3.9]{OurPaper} to prove that this implies that the event happens a.s. conditioned on non-extinction.
    To prove that our set contains interval with positive probability we will prove that $\Lambda_p$ contains the interval $J^{(U^*)}_{\pmb{\theta}^{*}}$ with positive probability, recall that $U^{*}$ and $\pmb{\theta}^{*}$ appeared in the first condition of the theorem. This we do by showing that the number of retained cylinders in this interval grows exponentially using the large deviation theorem (see Fact \ref{x78}) and the characterization as sums of independent random variables (see Lemma \ref{x41}). 

    Since the random variables $\mathbf{Z}_{n}^{(.)}(\overline{\pmb{\theta}})$ only depends on the first $n$ letter the word $\overline{\pmb{\theta}}$, $\overline{\pmb{\theta}}|_n$ hence we sometimes write $\mathbf{Z}_{n}^{(.)}(\overline{\pmb{\theta}}|_n)$ for the random variable.
    We define the events
    \begin{align*}
        A_0&:=\{\exists V\in [N]:\;Z^{(U^{*})}_{S^*}(\pmb{\theta}^{*})(V)>0\},\\
        A_n(\pmb{\theta})&:=\{\exists V\in [N]:\; Z^{(U^{*})}_{S^*+n}(\pmb{\theta}^{*}\pmb{\theta})(V)>0\}\text{ and }
        A_n:=\bigcap_{\pmb{\theta}\in[L]^n}A_n(\pmb{\theta}).
    \end{align*}
    If $\bigcap_{n}A_n$ happens with positive probability then we are done, since it means that with positive probability we retain everything in $J_{\pmb{\theta}^{*}}^{(U^{*})}$, hence $\Lambda_{\mathcal{S}, p}\bigcap J^{(U^{*})}_{\pmb{\theta}^{*}}=\Lambda_{\mathcal{S}}\bigcap J^{(U^{*})_{\pmb{\theta}^{*}}}$ which by the second part of the first assumption contains an interval. Instead of inspecting $A_n$ however we inspect another event. Fix $1<\eta<\gamma$ and for $\pmb{\theta}\in[L]^{n\cdot S}$ $\pmb{\theta}=\pmb{\theta}_1, \dots, \pmb{\theta}_n$ ($\pmb{\theta}_i\in [L]^{S}$) let 
    \begin{align*}
    B_n(\pmb{\theta})&:=\{\exists \mathbf{u}\in \mathcal{U}:\;
    \mathbf{Z}^{(U^{*})}_{S^*+k\cdot S}(\pmb{\theta}^*\pmb{\theta}_1\dots \pmb{\theta}_k)\geq \eta^k\mathbf{u}\text{ for }k\leq n\}\text{ and }\\
    B_n&:= \bigcap_{\pmb{\theta}\in [L]^{n\cdot S}}B_n(\pmb{\theta}).
    \end{align*}
    Since $\mathcal{U}$ does not contain the vector $\mathbf{0}$ it follows that $B_n\subset A_{n\cdot S}$, hence also $\bigcap_{n} B_n\subset \bigcap_{n}A_n$. Since $B_{n+1}\subset B_{n}$ we will inspect the following
    \begin{equation}\label{x92}
        \mathbb{P}(\bigcap_n B_n)= \lim_{n\to \infty} \mathbb{P}(B_{\ell})\prod_{k=\ell+1}^{n}\mathbb{P}(B_k|B_{k-1}),
    \end{equation}
    for some appropriate $\ell$ which we choose later.

    \begin{lemma}
    For all finite $\ell$ $\mathbb{P}(B_{\ell})>0.$
    \end{lemma}
    \begin{proof}
        For $B_0$ the statement follows from the first assumption of the theorem (see \eqref{x66}), since the probability that in the first $S^{*}$ level we retain every cylinder is positive, in which case $\mathbf{Z}^{(U^{*})}_{S^*}(\pmb{\theta}^{*})=\mathbf{e}_{U^{*}}^T\mathbf{B}_{\pmb{\theta}^{*}}\geq \mathbf{u*}$, and we are done.
        For $B_n$ we use similar argument, with positive probability we keep every cylinder until level $S^{*}+n\cdot S$ in which case for every $\pmb{\theta}\in[L^{n\cdot S}]$ we have that $\mathbf{Z}^{(U^{*})}_{S^*+Sn}(\pmb{\theta}^{*}\pmb{\theta})=\mathbf{e}_{U^{*}}^T\mathbf{B}_{\pmb{\theta}^{*}}\mathbf{B}_{\pmb{\theta}}\geq \mathbf{u}^T\mathbf{M}_{\pmb{\theta}^{*}}\geq \gamma^n\mathbf{v}\geq\eta^n\mathbf{v}$ for some $\mathbf{v}\in \mathcal{U}$, by the repeated application of the second assumption of the theorem.
    \end{proof}

Now we have to prove that there exists an $\ell$ such that
\begin{equation}
    \lim_{n\to \infty}\prod_{k=\ell+1}^{n} \mathbb{P}(B_k|B_{k-1})>0.
\end{equation}

    To prove that the probability in \eqref{x92} is greater than $0$, we consider $\mathbb{P}(\overline{B_k}|B_{k-1})$,
    where $\overline{B_k}$ is the complement of the event $B_k$.
    \begin{lemma}\label{x97}
        There exists a $0<\delta<1$ such that for all $k\geq 0$
        $$
        \mathbb{P}(\overline{B_k}|B_{k-1})\leq N^2 L^{S\cdot k} \delta^{\eta^{k-1}}.
        $$

    \end{lemma}
    Now we prove the statement assuming Lemma \ref{x97}.
    $$
    \lim_{n\to \infty}\prod_{k=\ell+1}^{n} \mathbb{P}(B_k|B_{k-1})\geq \lim_{n\to \infty}\prod_{k=\ell+1}^{n} (1-N^2L^{S\cdot k}\delta^{\eta^{k-1}}).
    $$
    Clearly we can choose $\ell$ in such a way that the product converges to a non-zero number.
\begin{comment}
    By Fact \ref{x72} we can choose $\ell$ in such a way that the product converges to a non-zero number, since
    $$
    N^2\sum_{k=1}^{\infty}L^{S\cdot k}\delta^{\eta^{k-1}}<\infty.
    $$
    \begin{fact}\label{x72}
        Consider a sequence $(a_n)$, with $0\leq a_n< 1$, then if $\sum_{n=1}^{\infty} a_n<\infty$ then $\prod_{n=1}^{\infty}(1-a_n)>0$.
    \end{fact}
    \begin{proof}
        First assume that $\sum a_n<1$, then
        $$
        \prod (1-a_n)=1-\sum p_n+ \sum_{m, n} p_n p_m- \dots,
        $$
        clearly the earlier terms are greater or equal to the later ones, hence 
        $$
        \prod (1-a_n)\geq 1-\sum p_n>0.
        $$
        When $\sum a_n<1$ does not hold, we consider a sufficiently small tale, and do the same thing again.
    \end{proof}
\end{comment}
\begin{proof}[Proof of Lemma \ref{x97}]

    We will consider the level $S^{*}+k\cdot S$ environments one by one. For simplicity we denote 
    $$
    h(k):=S^{*}+kS
    $$
    For a fix $\pmb{\theta}\in [L]^{kS}$, $\pmb{\theta}=\pmb{\theta}_1\dots \pmb{\theta}_k$,  (${\pmb{\theta}}_i\in[L]^{S}$) we denote 
    \begin{equation}
        \underline{\pmb{\theta}}^{-}:=\pmb{\theta}_1\dots \pmb{\theta}_{k-1}.
    \end{equation}
    The proof lies on the observation that for a fix $\pmb{\theta}$ as above, conditioned on $\mathbf{Z}^{(U^{*})}_{h(k)}(\pmb{\theta}^{*}\pmb{\theta}^-)$  the number of level $k\cdot S$ individuals 
    (the elements of the vector $\mathbf{Z}^{(U^{*})}_{h(k)}(\pmb{\theta}^{*}\pmb{\theta})$) can be written as a sum of independent random variables which are also independent of 
    $\mathbf{Z}^{(U^{*})}_{h(k)}(\pmb{\theta}^{*}\pmb{\theta}')$ for any $\pmb{\theta}'\in [L]^{S(k-1)}$ (see Lemma \ref{x41}). 
    \begin{align*}
        \mathbb{P}(\overline{B_k}|B_{k-1})
        &= \mathbb{P}( \{\exists 1\leq \ell\leq k, \, \exists \pmb{\theta}\in[L]^{\ell S}\, \forall \mathbf{u}\in \mathcal{U}\, \mathbf{Z}^{(U^{*})}_{h(\ell)}(\pmb{\theta}^{*}\pmb{\theta})\ngeq \gamma^\ell\mathbf{u}\}|B_{k-1}) \\ 
        &
        = \mathbb{P}( \{\exists \pmb{\theta}\in[L]^{kS}\, \forall \mathbf{u}\in \mathcal{U}\, \mathbf{Z}^{(U^{*})}_{h(k)}(\pmb{\theta}^{*}\pmb{\theta})\ngeq  \eta^k\mathbf{u}\}|B_{k-1})\\ 
        &
        \leq \sum_{\pmb{\theta}\in[L]^{S\cdot k}} \mathbb{P}( \{ \forall \mathbf{u}\in \mathcal{U}\, \mathbf{Z}^{(U^{*})}_{h(k)}(\pmb{\theta}^{*}\pmb{\theta})\ngeq  \eta^k\mathbf{u}\}|B_{k-1}) 
    \end{align*}
    The first equality if the definition on $\overline{B}_k$, the second one follows from the meaning of the condition $B_{k-1}$, that the event can't happen for $\ell<k$, and the last inequality is just the union bound.
    Fix an arbitrary $\pmb{\theta}\in [L]^{Sk}$, and inspect $\mathbb{P}( \{ \forall \mathbf{u}\in \mathcal{U}\, \mathbf{Z}^{(U^{*})}_{h(k)}(\pmb{\theta}^{*}\pmb{\theta})\ngeq\eta^k\mathbf{u}\}|B_{k-1})$.
    By Lemma \ref{x41}
    $$
    Z^{(U^{*})}_{S^*+Sn}(\pmb{\theta}^{*}\pmb{\theta})(W)
    =\sum_{V\in[N]} \sum_{\mathbf{i}\in\mathcal{X}^{(V)}_{*}(\pmb{\theta}^-)}
    Y^{(V)}_{\mathbf{i},\overline{\pmb{\theta}}|_{h(k)}, S}(W)=\sum_{V\in[N]} \sum_{j=1}^{Z^{(U^{*})}_{h(k-1)}(\pmb{\theta}^{*}\pmb{\theta}^-)(V)}
    Y^{(V)}_{j,\overline{\pmb{\theta}}|_{h(k)}, S}(W),
    $$
    where the random variables $Y^{(V)}_{j,\overline{\pmb{\theta}}|_{h(k)}}(W)$ are independent,distributed according to $Z^{(V)}_{S}(\pmb{\theta}_k)(W)$ and independent of $Z^{(U^{*})}_{h(k-1)}(\pmb{\theta}^{*}\pmb{\theta}^-)(V)$ for all $V\in [N]$ and $\pmb{\theta}'\in [L]^{S\cdot(k-1)}$ (including $\pmb{\theta}^{-}$).
    Recall that $B_{k-1}=\{\forall \pmb{\theta}\in [L]^{S\ell}\,\exists \mathbf{u}\in \mathcal{U}:\;\mathbf{Z}^{(U^{*})}_{h(\ell)}(\pmb{\theta}^{*}\pmb{\theta})>\eta^{\ell}\mathbf{u}\text{ for }\ell\leq k-1\}$, hence conditioned on $B_{k-1}$ for some $(v_0, \dots, v_{N-1})=\mathbf{v}\in \mathcal{U}$, we have that
    \begin{equation}\label{x29}
        Z^{(U^{*})}_{h(k)}(\pmb{\theta}^{*}\pmb{\theta})(W) =\sum_{V\in[N]} \sum_{j=1}^{Z^{(U^{*})}_{h(k-1)}(\pmb{\theta}^{*}\pmb{\theta}^-)(V)}
        Y^{(V)}_{j,\overline{\pmb{\theta}}|_{h(k)}, S}(W) \geq \sum_{\substack{V\in[N]\\v_V\neq 0}} \sum_{j=1}^{\ceil*{\eta^{k-1}v_V}}
        Y^{(V)}_{j,\overline{\pmb{\theta}}|_{h(k)}, S}(W).
    \end{equation}
    For this $\mathbf{v}$ and $\pmb{\theta}_k$ (by the second assumption of the theorem) there exists (at least one) $\widetilde{\mathbf{u}}$ such that 
\begin{equation}
    \eta^{k-1}\mathbf{v}^T\mathbf{M}_{\pmb{\theta}_k}\geq \eta^{k-1}\gamma\widetilde{\mathbf{u}}^T,
\end{equation}
hence we inspect
\begin{align*}
    &\mathbb{P}\big( \{ \forall \mathbf{u}\in \mathcal{U}\, \mathbf{Z}^{(U^{*})}_{h(k)}(\pmb{\theta}^{*}\pmb{\theta})\ngeq \eta^k\mathbf{u}\}|B_{k-1}\big)
    \leq \mathbb{P}\big( \{ \mathbf{Z}^{(U^{*})}_{h(k)}(\pmb{\theta}^{*}\pmb{\theta})\ngeq \eta^k\widetilde{\mathbf{u}}^T\}|B_{k-1}\big) \\
    &
    \leq \sum_{\substack{W\in[N]\\\widetilde{u}_W\neq 0}}\mathbb{P}\big( \{ \mathbf{Z}^{(U^{*})}_{h(k)}(\pmb{\theta}^{*}\pmb{\theta})<\eta^k\widetilde{u}_{W}\}|B_{k-1}\big)\\
    & \leq  \sum_{\substack{W\in[N]\\\widetilde{u}_W\neq 0}}\mathbb{P}\big( \{ \sum_{\substack{V\in[N]\\v_V\neq 0}} \sum_{k=1}^{\ceil*{\eta^{k-1}v_V}} Y^{(V)}_{j,\overline{\pmb{\theta}}|_{h(k)}, S}(W)< \eta^k\widetilde{u}_{W}\}|B_{k-1}\big)\\
    &=\sum_{\substack{W\in[N]\\\widetilde{u}_W\neq 0}}\mathbb{P}\big( \{ \sum_{\substack{V\in[N]\\v_V\neq 0}} \sum_{j=1}^{\ceil*{\eta^{k-1}v_V}} Y^{(V)}_{j,\overline{\pmb{\theta}}|_{h(k)}, S}(W)< \eta^k\widetilde{u}_{W}\}\big).
\end{align*}
The second inequality follows from the definition of $\ngeq$ and the union bound, the third is from \eqref{x29}, and the last inequality is the consequence of the independence of the summands from the condition. 
Now for $W$ such that $\widetilde{u}_W\neq 0$ we inspect the event 
$$
D_W=\bigg\{ \sum_{\substack{V\in[N]\\v_V\neq 0}} \sum_{j=1}^{\ceil*{\eta^{k-1}v_V}} Y^{(V)}_{j,\overline{\pmb{\theta}}|_{h(k)}, S}(W)< \eta^k\widetilde{u}_W\bigg\}.
$$
Recall that 

\begin{equation*}
    Y^{(V)}_{j,\overline{\pmb{\theta}}|_{h(k)}, S}(W)\sim Z^{(V)}_{S}(\pmb{\theta}_k)(W)\text{, hence } \mathbb{E}(Y^{(V)}_{j,\overline{\pmb{\theta}}|_{h(k)}, S})=\mathbb{E}(Z^{(V)}_{S}(\pmb{\theta}_k)(W))=\mathbf{M}_{\pmb{\theta}_k}(V, W).
\end{equation*}
Since 
$$
\sum_{\substack{V\in[N]\\v_V\neq 0}}\frac{1}{\gamma}v_V \mathbf{M}_{\pmb{\theta}_k}(V, W)\geq \widetilde{u}_W,
$$
we get that 
\begin{align*}
D_W
&\subset \bigg\{ \sum_{\substack{V\in[N]\\v_V\neq 0}} \sum_{j=1}^{\ceil*{\eta^{k-1}v_V}}Y^{(V)}_{j,\overline{\pmb{\theta}}|_{h(k)}, S}(W)\leq \eta^k\sum_{\substack{V\in[N]\\v_V\neq 0}}\frac{1}{\gamma}v_V \mathbf{M}_{\pmb{\theta}_k}(V, W)\bigg\}\\
&\subset \bigcup_{\substack{V\in[N]\\v_V\mathbf{M}_{\pmb{\theta}_k}(V, W)\neq 0}}\bigg\{\sum_{j=1}^{\ceil*{\eta^{k-1}v_V}}Y^{(V)}_{j,\overline{\pmb{\theta}}|_{h(k)},S}(W)\leq \eta^k\frac{1}{\gamma}v_V\mathbf{M}_{\pmb{\theta}_k}(V, W)\bigg\},
\end{align*}
hence
\begin{align*}
    &\sum_{\substack{W\in[N]\\\widetilde{u}_W\neq 0}}\mathbb{P}\big( \{ \sum_{\substack{V\in[N]\\v_V\neq 0}} \sum_{j=1}^{\ceil*{\eta^{k-1}v_V}}Y^{(V)}_{j,\overline{\pmb{\theta}}|_{h(k)}, S}(W)< \eta^k\widetilde{u}_W\}\big)\\
    &\leq \sum_{\substack{W\in[N]\\\widetilde{u}_W\neq 0}}
    \sum_{\substack{V\in[N]\\v_V\mathbf{M}_{\pmb{\theta}_k}(V, W)}\neq 0}\mathbb{P}\big(\{\sum_{j=1}^{\ceil*{\eta^{k-1}v_V}}Y^{(V)}_{j,\overline{\pmb{\theta}}|_{h(k)}, S}(W)\leq \eta^{k-1}v_V\frac{\eta}{\gamma} \mathbf{M}_{\pmb{\theta}_k}(V, W)\}\big).
\end{align*}
The summands here are i.i.d random variables with expectation $0<\mathbf{M}_{\theta_{k}}(V,W)$, and since $\eta<\gamma$, we have $\frac{\eta}{\gamma}\mathbf{M}_{\pmb{\theta}_k}(V,W)<\mathbb{E}(Y^{(V)}_{j,\overline{\pmb{\theta}}|_{h(k)}, S}(W))$, hence we can use the large deviation lemma (Fact \ref{x78}), to get that there exists a $0<\delta(\pmb{\theta}_k, V,W)<1$ such that
\begin{multline*}
    \mathbb{P}\big(\{\sum_{j=1}^{\ceil*{\eta^{k-1}v_m}}Y^{(V)}_{j,\overline{\pmb{\theta}}|_{h(k)}, S}(W)\leq \eta^{n-1}v_V\frac{\eta}{\gamma} \mathbf{M}_{\pmb{\theta}_k}(V,W)\}\big)\leq \delta(\pmb{\theta}_k, V,W)^{\eta^{k-1}v_V}
    \\
    \leq \delta(\theta_k, V,W)^{\eta^{n-1}},
\end{multline*}
where the last inequality follows from the fact, that $v_V\geq1$ whenever $v_V\neq 0$.

Let 
\begin{equation*}
\delta=\max\left\{ \delta(\pmb{\theta}_k, V, W): \; \pmb{\theta}_k\in[L]^{S}, \; V, W\in [N]\;\;\mathbf{M}_{\pmb{\theta}_k}(V,W)\neq 0 \right\}<1.
\end{equation*}
Then 
\begin{equation*}
    \sum_{\substack{W\in[N]\\\widetilde{u}_W\neq 0}}
    \sum_{\substack{V\in[N]\\v_V\mathbf{M}_{\pmb{\theta}_k}(V, W)}\neq 0}\mathbb{P}\big(\{\sum_{j=1}^{\ceil*{\eta^{k-1}v_m}}Y^{(V)}_{j,\overline{\pmb{\theta}}|_{h(k)}, S}(W)\leq \eta^{k-1}v_V\frac{\eta}{\gamma} \mathbf{M}_{\theta_k}(V, W)\}\big)\leq N^2 \delta^{\eta^{k-1}},
\end{equation*}
from which we get that
\begin{equation*}
    \mathbb{P}\big(\overline{B_k}|B_{k-1}\big)\leq L^{S\cdot k}N^2 \delta^{\eta^{k-1}}.
\end{equation*}

\end{proof}
This proves Lemma \ref{x97}, and since we already proved the statement assuming the lemma, this also proves the statement.
\end{proof}

\begin{appendix}
    \section{MBPREs}\label{app1}
This section is the continuation of Section \ref{x23}, we use the notation from therein.
For each $i\in[N]$ and $\theta\in[L]$ there is an offspring vector random variable
\begin{equation*}
\mathbf{Y} _{\theta}^{(i)}=(Y _{\theta}^{(i) }(0), \dots ,Y _{\theta}^{(i)}(N-1))
 \end{equation*}
which is distributed according to $f_{\theta}^{(i)}$, i.e.
 \begin{equation*}
 \mathbf{P}\left(\mathbf{Y}_{\theta}^{(i)}=\mathbf{y}\right)=f _{\theta}^{(i) }[\mathbf{y}] ,\, \text{ for every } \mathbf{y}\in\mathbb{N} _{0}^{N}.
\end{equation*}
Now for a fix $\overline{\pmb{\theta}}\in \Sigma$ we define $\left\{ \mathbf{Z}_n(\overline{\pmb{\theta}})\right\}_{n\geq 0}$ the \textit{N-type branching process in the varying environment}  $\overline{\pmb{\theta}}=(\theta_1,\theta_2, \dots)$ (sometimes called time inhomogeneous multitype branching process).
We start at level $0$, where the number of different types of individuals is deterministic and is given by
$\mathbf{z}_0:=(z _{0}^{(0) },\dots ,z _{0}^{ (N-1)})$, that is
$\mathbf{Z}_0(\overline{\pmb{\theta}})=\mathbf{Z}_0:=\mathbf{z}_0$. Given $\mathbf{Z}_0(\overline{\pmb{\theta}}), \dots, \mathbf{Z}_{n-1}(\overline{\pmb{\theta}})$ we define $\mathbf{Z}_n(\overline{\pmb{\theta}})$ as follows.

We consider the sequence of vector random variables
          $$
          \left\{\mathbf{Y}_{j,\theta_n}^{(i)}
              =(Y _{j,\theta_n}^{(i) }(0),\dots Y _{j,\theta_n}^{(i) }(N-1)) : i\in [N],\ j\in\{1,\dots, \mathbf{Z}_{n-1}^{(i)}(\overline{\pmb{\theta}})\}\right\},
          $$
          
          \begin{enumerate}[label={(\alph*)}]
            \item $\left\{\mathbf{Y}_{j,\theta_n}^{(i)}\right\}_{i,j}$ are independent of each other and $\mathbf{Z}_{n-1}$, and 
            \item $\mathbf{Y}_{j,\theta_n}^{(i)}\stackrel{d}{=}\mathbf{Y}_{\theta_n}^{ (i)}$.
          \end{enumerate}

          Informally the meaning of the $\ell $-th component,
          $Y_{j,\theta_n}^{(i)}(\ell )$ of $\mathbf{Y}_{j,\theta_n}^{(i)}$ is the number of type-$\ell$ individuals of level $n$ given birth by the $j$-th level $n-1$ type-$i$ individuals.
          
Then the vector of the numbers of various type level-$n$ individuals is
\begin{equation} \label{eq:z72}
\mathbf{Z}_n(\overline{\pmb{\theta}})=(Z_n^{(0)}(\overline{\pmb{\theta}}), \dots, Z_{n}^{(N-1)}(\overline{\pmb{\theta}})):= \sum_{i=0}^{N-1}\sum_{j=1}^{Z_{n-1}^{(i)}(\overline{\pmb{\theta}})}\mathbf{Y}_{j,\theta_n}^{(i)},
\end{equation}
where $Z _{n}^{(i)}(\overline{\pmb{\theta}})$ stands for the number of type $i$ individual in the $n$-th generation.

The corresponding MBPRE is the process $\mathcal{Z}=\{\mathbf{Z}_n\}_{n=1}^{\infty}$ satisfying that 
 for each environment $\overline{\pmb{\theta}}$ chosen according to the measure $\nu$ and for each $\mathbf{z}_0, \mathbf{z_1}, \dots, \mathbf{z_k}\in \mathbb{N}_{0}^N$,
	\begin{multline*}
		\mathbb{P}(\mathbf{Z}_1=\mathbf{z}_1, \dots, \mathbf{Z}_k=\mathbf{z}_k|\mathbf{Z}_0=\mathbf{z}_0, \mathcal{V}=\overline{v})
		= \mathbf{P}(\mathbf{Z}_1(\overline{\pmb{\theta}})=\mathbf{z}_1, \dots, \mathbf{Z}_k(\overline{\pmb{\theta}})=\mathbf{z}_k)\text{  a.s.}.
	\end{multline*}
 We write $\mathbb{P}\left(\cdot\right)$ and $\mathbb{E}\left(\cdot\right)$ for the probabilities and expectations in random environments.
 For each $\theta\in[L]$ the $N\times N$ expectation matrices are  
 \begin{equation}\label{x19}
    \mathbf{M}_{\theta}(i, j)= \mathbb{E} (Y^{(i)}_{\theta}(j)), \; i, j\in[N].
 \end{equation}
 \section{Higher dimensions}\label{app2}

\subsection*{The IFS}
We consider the $d$-dimensional IFS---analogously to the one dimensional \eqref{a99}---of the following form:
\begin{align}\label{x11}
    &\mathcal{S}:=\left\{
        S_i(x):=\frac{1}{L}\mathbf{x}+\mathbf{t}_i
     \right\}_{i=0}^{M-1}
     , S_i:\mathbb{R}^{d}\to\mathbb{R}^{d},
    \\
    &L\in\N \setminus \left\{0,1\right\},
    \, \mathbf{t}_i\in \mathbb{N}^{d},\, \exists h\in \mathbb{N},\,L-1|h,\,\forall j\in[d]\, 0=\min_{i}\mathbf{t}_i(j), \, h=\max_{i}\mathbf{t}_i(j).
    \end{align}
We denote $I=[0,\frac{h}{L-1}L]^{d}$, and for $i_1, \dots,i_d \in[h/(L-1)-1]$
\begin{equation}
        J^{(i_1, \dots, i_d)}:=[i_1\cdot L, (i_1+1)\cdot L]\times \dots \times[i_d\cdot L, (i_d+1)\cdot L].
\end{equation}
Now we consider the natural measure as in \eqref{a94}, and those intervals $J^{(i_1, \dots, i_d)}$ for which $\nu(J^{(i_1, \dots, i_d)})>0$. We arrange them in lexicographical order to get the basic cubes $J^{(0)}, \dots, J^{(N)}$ analogously to the one-dimensional case. The vertex of $J^{(k)}$ closest to the origin is denoted by $\mathbf{b}_k=(b_k(0),\dots, b_k(d-1))$.
We consider the $L^d$-adic cubes inside the basic cubes, arranged in lexicographical order: For $(\theta_0, \dots, \theta_{n-1})=\pmb{\theta} \in [L^d]^n$ we consider 
$((\widetilde{\theta}_0(0), \dots, \widetilde{\theta}_{0}(n-1)), \dots, (\widetilde{\theta}_{d-1}(0), \dots, \widetilde{\theta}_{d-1}(n-1)))=\widetilde{\pmb{\theta}}\in ([L]^{n})^{d}$ such that for $i\in[n]$ $\theta_i= L^{d-1}\cdot\widetilde{\theta}_{0}(i)+ \dots+\widetilde{\theta}_{d-1}(i)=\sum_{k=0}^{d-1}L^{k}\theta_{d-1-k}(i)$. Then

\begin{equation}
    J^{(N)}_{\pmb{\theta}}=\bigtimes_{k=0}^{d-1}
        \left[b_k(k)L+\sum_{\ell =1}^{n}
        \widetilde{\theta}_k (\ell) L^{-(\ell -1)}  ,
        b_k(k)L+\sum_{\ell =1}^{n }
        \widetilde{\theta}_k (\ell) L^{-(\ell -1)}+L^{-(n-1)} \right]
\end{equation}
\begin{comment}
\begin{align*}
    J^{(N)}_{\pmb{\theta}}=&
        \left[b_k(0)L+\sum_{\ell =1}^{n}
        \widetilde{\theta}_0 (\ell) L^{-(\ell -1)}  ,
        b_k(0)L+\sum_{\ell =1}^{n }
        \widetilde{\theta}_0 (\ell) L^{-(\ell -1)}+L^{-(n-1)} \right]
         \times \dots \times\\&
         \left[
            b_k(d-1)L+\sum_{\ell =1}^{n}
            \widetilde{\theta}_{d-1} (\ell) L^{-(\ell -1)}  ,
            b_k(d-1)L+\sum_{\ell =1}^{n }
            \widetilde{\theta}_{d-1} (\ell) L^{-(\ell-1)}+L^{-(n-1)}\right].
\end{align*}
\end{comment}
We can again describe these systems using the matrices from \eqref{a90}, namely for $\theta\in [L^{d}]$ and $i,k\in [N]$, 

\begin{equation}
        \mathbf{B}_{\theta} (i,k) :=
        \#\left\{ \ell \in [M]:
        S_\ell (J^{(k)})=J _{\theta }^{(i) }
        \right\}.
\end{equation}

The randomization of such sets happens just as in Section \ref{a77} for a given parameter $p\in(0,1]$ and $\mathcal{S}$ as above. We refer to the resulting random IFSs as $d$-dimensional CISSIFSs, and will denote the particular attractors by $\Lambda^{(d)}_{\mathcal{S}, p}$. Again the corresponding expectation matrices $\mathbf{M}_{\theta} =p\cdot \mathbf{B}_{\theta}$.

The three main statements of the paper holds in higher dimensions.
\begin{proposition}
    Consider the attractor $\Lambda_{\mathcal{S}, p}$ of the $d$-dimensional CISSIFS and the corresponding $N\times N$ expectation matrices $\mathcal{M}=\{\mathbf{M}_0, \dots, \mathbf{M}_{L^d-1}\}$ ($\mathcal{B}=\{\mathbf{B}_i=p^{-1}\mathbf{M}_i\}$). Assume that they are good in the sense of Definition \ref{x52} with respect to the uniform measure ($\nu=(L^{-d}, \dots, L^{-d})^{\N}$) on $\Sigma=[L^d]^{\N}$. Denote $\lambda$ the Lyapunov exponent of $\mathcal{B}$ with respect to $\nu$.

    Then
    \begin{enumerate}
        \item the conclusions of Theorem \ref{z42} hold for a positive constant multiple of the $d$-dimensional Lebesgue measure.
        \item Proposition \ref{x94}, describing the almost sure empty interior holds.
        \item Proposition \ref{x98}, describing the existence of an interior point in the attractor holds.
    \end{enumerate} 
\end{proposition}

\begin{proof}
    The proof is almost character to character agrees to the ones in the one-dimensional case. 
\end{proof}
\begin{remark}
    Here we only consider the question of existence of interior points, however in higher dimensions it is a different question whether the attractor is totally disconnected or not, which question we do not address here. However, in case of the example above it is clear that if $p$ is big enough then there exists a curve connecting the left and the right wall with positive probability, since the system contains the $2\times 2$ Mandelbrot percolation as a subsystem, which is known to percolate with positive probability for large enough $p$ (for the current best bound see for example \cite{henk_don_phd}). 
    We can ask the usual questions regarding percolation in $2$-dimensional CISSIFSs, the first natural one is if it is true that the set either percolates with positive probability or is totally disconnected with probability one (conditioned on non-extinction).
    Another, general question is the almost sure Hausdorff (which agrees to the box-dimension by \cite[Theorem 3.5]{zbMATH06808299}) dimension of such sets.
\end{remark}
\begin{figure}
    \centering
    \includegraphics[width=0.33\linewidth]{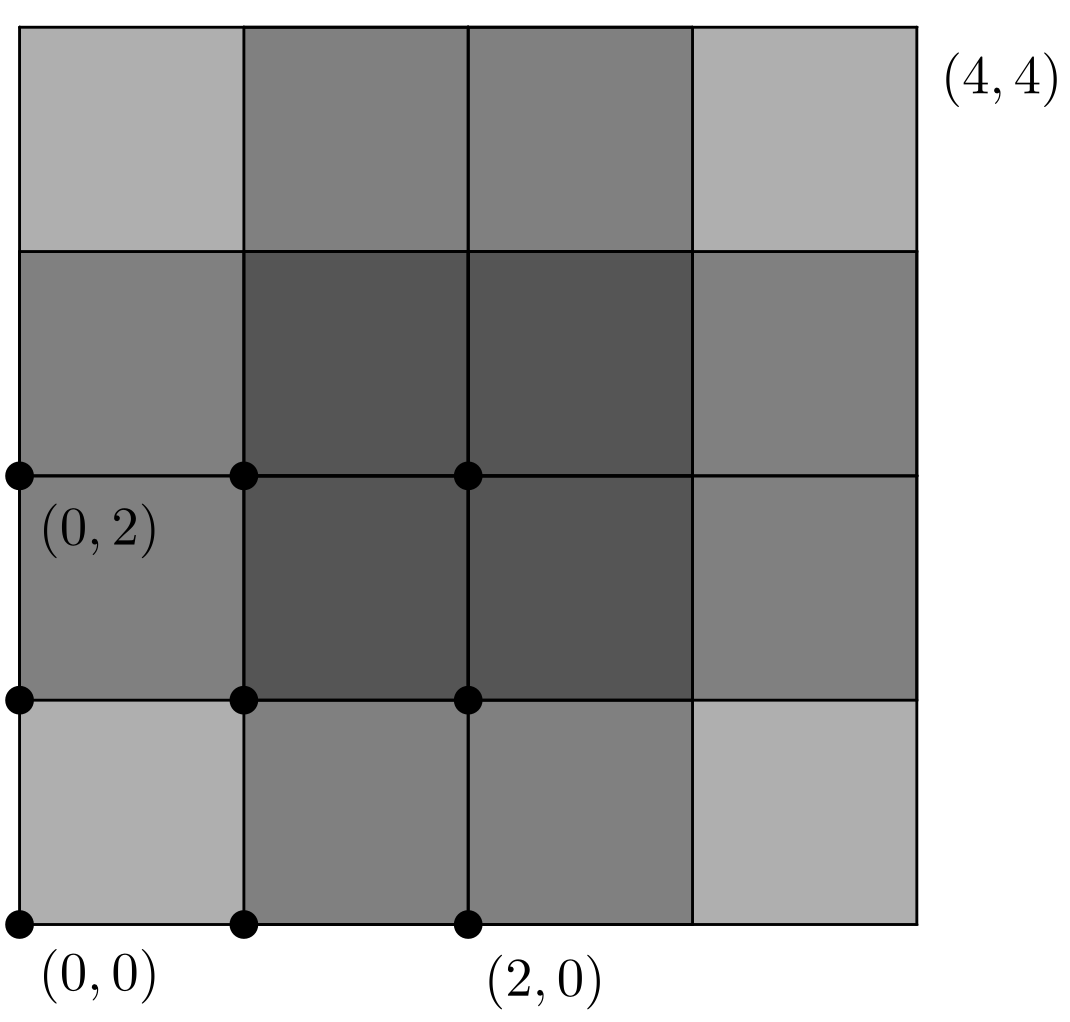}\label{y99}
    \includegraphics[width=0.36\linewidth]{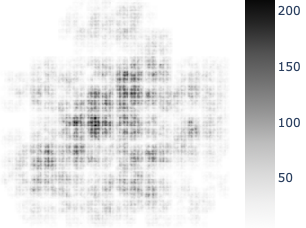}\label{y98}
    \includegraphics[width=0.28\linewidth]{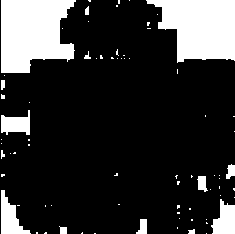}\label{y97}
    \caption{On the first figure the first level approximation of the IFS in Example \ref{x02} is depicted. The $\{\mathbf{t}_i\}_{i=1}^{9}$ translations are denoted by black points on the figure. On the second and third the level $6$ approximation of the random set with parameter $p=0.7$ is shown. On the first the darkness depicts the number of cylinders covering a given square, on the second everything that is retained up to the $6$-th approximation is black.}
  \end{figure}
\begin{example}[Overlapping Mandelbrot percolation]\label{x02}
    The following is one of the simplest 2-dimensional examples.
    \begin{equation}
        \mathcal{S} =\left\{\frac{\mathbf{x}}{2}+\mathbf{t}_i\right\}_{i=1}^{9},
    \end{equation}
    where $\mathbf{t}_i$ runs through the set $\{0, 1, 2\}^2$.
\end{example}
\begin{lemma}
    In the Overlapping Mandelbrot percolation example (Example \ref{x02}) we get the following bounds:
    \begin{itemize}
        \item $\Lambda_p$ contains a ball almost surely conditioned on non extinction iff $p=1$.
        \item When $p>0.993$ then by \cite{henk_don_phd} the set contains a curve which connects the left and right walls with positive probability.
        \item When $p>0.7712$ then the set has positive two dimensional Lebesgue measure almost surely conditioned on non-extinction.
    \end{itemize}
\end{lemma}
In this case $N=4$, and the corresponding matrices are:
    \begin{align*}
        \mathbf{B}_0=
            \begin{bmatrix}
             1 & 0 & 0 & 0 \\
             1 & 1 & 0 & 0 \\
             1 & 0 & 1 & 0 \\
             1 & 1 & 1 & 1 \\
            \end{bmatrix},\,
        \mathbf{B}_1=
                \begin{bmatrix}
                 1 & 1 & 0 & 0 \\
                 0 & 1 & 0 & 0 \\
                 1 & 1 & 1 & 1 \\
                 0 & 1 & 0 & 1 \\
                \end{bmatrix},\,
        \mathbf{B}_2=
                \begin{bmatrix}
                    1 & 0 & 1 & 0 \\
                    1 & 1 & 1 & 1 \\
                    0 & 0 & 1 & 0 \\
                    0 & 0 & 1 & 1 \\
                   \end{bmatrix},\,
                   \mathbf{B}_3=
                   \begin{bmatrix}
                    1 & 1 & 1 & 1 \\
                    0 & 1 & 0 & 1 \\
                    0 & 0 & 1 & 1 \\
                    0 & 0 & 0 & 1 \\
                   \end{bmatrix}.
    \end{align*}
    All of them are allowable.
    Since the spectral radius of $\mathbf{B}_0=1$ it follows that the lower spectral radius can only be exactly equal to 1 (since the matrices are all allowable, hence their product as well, and are non-negative integer matrices), hence whenever $p<1$ the interior of the attractor is almost surely empty.
    We inspect the positivity of the Lebesgue measure. We only give a crude estimation for the Lyapunov exponent, to show that $\lambda>0$, implying that there exists an interval where no interior point exists but the Lebesgue measure is positive. For this we use a Theorem of Hennion (\cite[Theorem 2]{Hennion97}) according to which the Lyapunov exponent $\lambda=\lim_{n\to\infty}1/n\log (\mathbf{M}_{\theta_1}\dots\mathbf{M}_{\theta_n})_{*}$, where
    $(\mathbf{B})_{*}$ is the smallest column sum of the matrix $\mathbf{B}$. We estimate $\lambda$, by using the fact that for all $\mathbf{B}_1, \mathbf{B}_2$ non-negative, allowable matrices $(\mathbf{B}_1\cdot\mathbf{B}_2)_{*}\geq (\mathbf{B}_1)_{*}\cdot(\mathbf{B}_2)_{*}$.
    From this it follows that
\begin{equation*}
\lambda\geq\lim_{(k\cdot m)\to\infty}\frac{1}{k\cdot m}\log \left[(\mathbf{M}_{\theta_1}\cdots \mathbf{M}_{\theta_m})_{*}\cdots (\mathbf{M}_{\theta_{(k-1)m+1}}\cdots\mathbf{M}_{\theta_{k\cdot m}})_{*}\right].
\end{equation*}
\end{appendix}
This, by the law of large numbers tends to $1/m\mathbb{E}(\alpha_m)$, where the random variable $\alpha_m=\log(\mathbf{M}_{\theta_1}\dots \mathbf{M}_{\theta_m})_{*}$, where $\theta_1\dots\theta_m$ is chosen uniformly from $[4]^{m}$. 

In particular for $m=2$ it can be calculated (using for example Wolfram Mathematica) that $\mathbb{P}(\alpha_2=0)=1/4$, hence $1/2 \mathbb{E}(\alpha_2)\geq 3/8\log 2>0$. This gives the estimation for the critical value above which the Lebesgue measure is positive almost surely conditioned on non-extinction $p^{*}\leq 0.7712$. Using larger $m$ we can possibly get a better estimation for the critical value $p^{*}$, so that for $p>p^{*}$ $\Lambda_{\mathcal{S},p}$ has positive Lebesgue measure almost surely conditioned on non-extinction.

\begin{acks}[Acknowledgments]
  We would like to say thanks to Reza Mohammadpour and to Alex Rutar for all the useful discussions.
\end{acks}

\begin{funding}
VO is supported by National Research,
Development and Innovation Office - NKFIH, Project FK134251.
	
KS is supported by National Research,
Development and Innovation Office - NKFIH, Project K142169 and NKFI KKP144059.

Both authors have received funding from the HUN-REN Hungarian Research
Network.
\end{funding}

\bibliographystyle{imsart-nameyear} 
\bibliography{references}  

\end{document}